\documentclass{amsart}
\usepackage{amsmath,amscd, amssymb}
\usepackage{caption}
\usepackage{color}
\usepackage{enumitem}
\usepackage[bookmarksnumbered=true]{hyperref}
\hypersetup{colorlinks = true, citecolor = blue}
\usepackage{natbib}
\usepackage{stmaryrd}
\usepackage{xcolor,colortbl}
\usepackage[utf8]{inputenc}
\usepackage{color}
\usepackage{soul}
\usepackage{float}
\usepackage{subcaption}
\usepackage{graphics}
\graphicspath{{Figures/}}
\usepackage[most]{tcolorbox}
\usepackage{lineno}
\usepackage{pgf,tikz}
\usetikzlibrary{arrows,calc,decorations.pathreplacing,angles,arrows.meta,positioning,shapes}
\usepackage{blkarray}

 \usepackage{mhchem} 

\newcommand{\BF}{\boldmath}
\usepackage{arydshln}

\usepackage{accents}
\newcommand*{\dt}[1]{%
  \accentset{\mbox{\large\bfseries .}}{#1}}
\newcommand*{\ddt}[1]{%
  \accentset{\mbox{\large\bfseries .\hspace{-0.25ex}.}}{#1}}
\newcommand*{\dddt}[1]{%
  \accentset{\mbox{\large\bfseries .\hspace{-0.25ex}.\hspace{-0.25ex}.}}{#1}}

\newcommand{\beq}{\begin{linenomath}\begin{equation}\begin{aligned}} 
\newcommand{\eeq}{\end{aligned}\end{equation}\end{linenomath}} 

\usepackage{graphicx}

\makeatletter
\newcommand{\R}{{\mathbb{R}}}
\newcommand*\bigcdot{\mathpalette\bigcdot@{.5}}
\newcommand*\bigcdot@[2]{\mathbin{\vcenter{\hbox{\scalebox{#2}{$\m@th#1\bullet$}}}}}
\makeatother

\usepackage{graphicx,scalerel}

\tikzset{bluenode/.style={draw,circle,scale=.8,fill=blue!40}}
\tikzset{rednode/.style={draw,circle,scale=.8pt,fill=red!50}}
\tikzset{whitenode/.style={draw,circle,scale=.8pt,fill=white!50}}

\makeatletter
\tikzset{circle split part fill/.style  args={#1,#2}{
 alias=tmp@name,
  postaction={%
    insert path={
     \pgfextra{%
     \pgfpointdiff{\pgfpointanchor{\pgf@node@name}{center}}%
                  {\pgfpointanchor{\pgf@node@name}{east}}%
     \pgfmathsetmacro\insiderad{\pgf@x}
       \fill[#1] (\pgf@node@name.base) ([xshift=-\pgflinewidth]\pgf@node@name.east) arc
                          (0:180:\insiderad-\pgflinewidth)--cycle;
      \fill[#2] (\pgf@node@name.base) ([xshift=\pgflinewidth]\pgf@node@name.west)  arc
                           (180:360:\insiderad-\pgflinewidth)--cycle;            
         }}}}}  
 \makeatother  
 \tikzset{bnecknode/.style={shape=circle split,
    draw=white!40,dashed,
    line width=.0000000001mm,text=black,font=\bfseries,minimum size=.7cm,
    circle split part fill={blue!50,red!30}}}

\usetikzlibrary{shapes}

\newtheorem{theorem}{Theorem}

\newtheorem{cor}[theorem]{Corollary}
\newtheorem{defn}[theorem]{Definition}

\newtheorem{lemma}[theorem]{Lemma}
\newtheorem{Fact}[theorem]{Fact}
\newtheorem{prop}[theorem]{Proposition}
\numberwithin{equation}{section}
\newtheorem{example}{Example}

\definecolor{jim}{rgb}{1, .5, 0.0}

\definecolor{jim2}{rgb}{.6,.4,0}

\definecolor{Gray}{gray}{0.85}
\definecolor{dgn}{rgb}{0.0, 0.5, 0.0}
\newcommand{\black}{\color{black}{}}

\newcommand{\Supp}{\mathrm{supp}}
\newcommand{\ie}{\textit{i.e.}}

\newcommand{\dis}{{\displaystyle}}
\newcommand{\cH}{{\mathcal H}}
\newcommand{\cK}{\mathcal K}

\newcommand{\cP}{{\mathcal P^{(>0)}}}

\newcommand{\bcP}{{\mathcal P^{(\ge 0)}}}
\newcommand{\cQ}{\mathcal Q_{cc}}
\newcommand{\cT}{\mathcal T}
\newcommand{\SLV}{S_{LV}}
\newcommand{\SZ}{S_Z}

\DeclareSymbolFont{boldoperators}{OT1}{cmr}{bx}{n}
\SetSymbolFont{boldoperators}{bold}{OT1}{cmr}{bx}{n}
\edef\bar{\unexpanded{\protect\mathaccentV{bar}}\number\symboldoperators16}

\usepackage{tcolorbox}   
\newtcbox{\HL}[1][green]{on line, arc=7pt,colback=#1!10!white,colframe=#1!50!black,
  before upper={\rule[-3pt]{0pt}{10pt}},boxrule=1pt, boxsep=0pt,left=6pt,
  right=6pt,top=2pt,bottom=2pt}

\newtheorem{innercustomgeneric}{\customgenericname}
\providecommand{\customgenericname}{}
\newcommand{\newcustomtheorem}[2]{%
  \newenvironment{#1}[1]
  {%
   \renewcommand\customgenericname{#2}%
   \renewcommand\theinnercustomgeneric{##1}%
   \innercustomgeneric
  }
  {\endinnercustomgeneric}
}

\newcustomtheorem{customthm}{Theorem}
\newcustomtheorem{customlemma}{Lemma}

\definecolor{COM}{rgb}{0.4,0,0.4}

\definecolor{Ng}{rgb}{.5,0.15,.02}

\usepackage[normalem]{ulem}
\usepackage{babel}
\usepackage{lipsum}
\makeatletter
\g@addto@macro{\endabstract}{\@setabstract}
\newcommand{\authorfootnotes}{\renewcommand\thefootnote{\@fnsymbol\c@footnote}}%
\makeatother

\begin{document}
\title[A team of Die-out Lyapunov functions]
{{
Extinction of multiple populations and a team of Die-out Lyapunov functions}
}
\maketitle
\begin{center}
\authorfootnotes
  Naghmeh Akhavan\textsuperscript{1}, James A. Yorke\textsuperscript{2} \par \bigskip
  \textsuperscript{1} Department of Mathematics and Statistics, \ University of Maryland, Baltimore County, MD \par
  \textsuperscript{2} Departments of Mathematics and Physics, Institute for
Physical Science and Technology, University of Maryland,
College Park, MD \par \bigskip

\end{center}
\footnotetext[1]{nakhava1@umbc.edu}
\footnotetext[2]{yorke@umd.edu}
 
\maketitle
\begin{abstract}
The extinction of species is a major problem of concern with a large literature.
Our investigation gives insight into when species extinctions must occur, with an emphasis on determining which species might possibly die out and on how fast they die out.
We investigate a differential equations model for population interactions with the goal of determining when several species (\ie, coordinates of a bounded solution) must die out or ``go extinct'' and must do so exponentially fast. Typically each coordinate represents the population density of a different species.
For our main tool, we create what we call  ``die-out'' Lyapunov functions. A given system may have several or many such functions, each of which is a function of a different set of coordinates. 
That die-out function implies that one of the species in its subset must die out exponentially fast -- for almost every choice of coefficients of the system. 
We create a ``team'' of die-out functions that work together to show that $k$ species must die,  where $k$ is determined separately.
Secondly, we present a ``trophic'' condition for generalized Lotka-Volterra systems that guarantees that there is a trapping region that is globally attracting. That implies that all solutions are bounded.
\end{abstract}

\maketitle

\section{Introduction}

In this paper, we investigate two questions and two systems. First, we ask if we can tell when all solutions of the following (generalized) Lotka-Volterra are bounded. We show, under natural ``trophic'' assumptions below, that all solutions are bounded (and wind up in a bounded trapping region). Here the equations are: 

\begin{equation}\label{eq:LV}
  x_i'  = x_i \Big(c_i +\sum_{j=1}^{d} s_{ij}x_j\Big),\quad \text{where } i=1,\ldots,d, \text{ and }x_i' = \frac{d}{dt}{x_i},\ x_i\ge 0.
\end{equation}
Write $X:=(x_1,\ldots,x_d), C:=(c_1,\ldots,c_d)$ and $S:=(s_{ij})$. Then we can rewrite the equation, using an inner product, as
\begin{equation}\label{eq:LV1}
X' = X\cdot\Big(C + SX\Big)
\end{equation}
\begin{defn}\label{trophic-def}
We say a system Eq.~\eqref{eq:LV} (or the matrix $S:=(s_{ij})$) is {\bf trophic} if $S$ and $C:=(c_1,\ldots,c_d)$ satisfy the following two hypotheses.

\noindent
\begin{enumerate}
    \item [(T1)]  For all $n$ if $c_n \ge 0$, then $s_{nn}<0$.
    \item [(T2)]  For each pair $n,m$,
    if $s_{nm} >0$, then $n>m$ and furthermore $s_{mn}<0$.
\end{enumerate}
\end{defn}
\bigskip
An interpretation of  (T2) is that if species n benefits from species $m$, then $n>m$ and species $n$ must negatively affect the population density of species $m$. It also implies that $s_{mm}\le 0$ for all $m$, so that species can be self-limiting, \ie, $s_{mm}< 0$. We created this definition to reflect the term trophic and imply the existence of a trapping region. The reader may be able to find a more inclusive term for which Theorem \ref{thm:trap} below remains true.

 If {\it{(T1)}} is false, then $c_n>0$ and $s_{nn} \ge 0$ for some $n$. 
 Hence, if $X(t)$ is a solution with $x_m(t) \equiv 0$ for all $m \neq n$, then $x_n(t) \to \infty$ as $t \to \infty$.
Condition {\it{(T2)}} implies $s_{nn}$ can not be positive since  $s_{nm} > 0$ implies $m \neq n$. 

\begin{defn}
Let 
$\bcP$ be the set of points whose coordinates are non-negative, \ie, the domain of our differential equations,
 $ \{(x_1, \ldots, x_d):  x_i \ge0\}$.
We say $\Gamma$ is a {\textbf{globally attracting trapping region}} (for $\bcP$) if $\Gamma$ is bounded and every trajec\-tory $X(t)$ in $\bcP$ is eventually in $\Gamma$ and if $X(t_0) \in \Gamma$, then $X(t) \in \Gamma$ for all $t \ge t_0$, \cite{meiss2007differential}. 
\end{defn}

\begin{customthm}{1}
[\bf{``Trophic'' implies existence of a bounded globally attracting trapping region}]\label{thm:trap}
Assume the system Eq.~\eqref{eq:LV} is trophic; (see Def.~\ref{trophic-def}).
Then \eqref{eq:LV} has a bounded globally attracting trapping  region.
\end{customthm}

The proof of Theorem~\ref{thm:trap} is in section \ref{sec:trapping region}.
Now that we have a way of guaranteeing that some systems have bounded solutions, we address a more general family, and we ask when we can show that a bounded solution has species (\ie, coordinates that must die out.

\bigskip

There is an old ``competitive exclusion principle'', a rule of thumb for ordinary differential equations models of an ecosystem (\cite{cushing2004some}, \cite{yoon2021global}, \cite{yang2021estimating}), that if $d$ of species depend on $d'$ resources  when $d>d'$, there is no positive attracting steady state, \cite{levin1970community}.
 
Does that mean that $d-d'$ species die out? We explore situations where species must die out and must do so exponentially fast. Our results extend to cases where $d'$ is not less than $d$.
 Our results have implications well beyond ecosystems. See for example the extensive discussion of applications in \cite{li2010global}.

{\bf Our (non-autonomous) ``island" or ``resource" or} main model.
(We sometimes write ``$:=$'' instead of ``$=$'' in definitions).
Write $Z := (z_1, \ldots, z_{d'})\in\R^{d'}$,  and let $\SZ$ or $S = (s_{ij})$ be a 
$d \times d'$ matrix where $d,d'$ are positive integers. In particular, we may have $d' < d$, $d'=d$, or $d'>d$. When $d'<d$, For almost every $C$, each bounded solution $X(t)$ will have at least one species dying out, but some species may also die out in the other cases, depending upon the choice of matrix $S$. Our generalization of Eqs. \eqref{eq:Z} and \eqref{eq:Z1} are
\begin{equation}\label{eq:Z}
  x_i'(t) 
  = x_i(t)\Big(c_i +\sum_{j=1}^{d'} s_{ij}z_j(t)\Big),\quad \text{where } i=1,\ldots,d;
\end{equation}
where
$z_j(t)$ is a continuous function for each $j$, and each $z_j(t)$ may depend on $X(t)$ or even on $X(t-\tau_j)$ where the $\tau_j$ are time delays.  We can rewrite the system of equations as
\begin{equation}\label{eq:Z1}
X' = X\cdot\Big(C + SZ\Big)
\end{equation}

 Each $x_i$ can be viewed as the population density of the $i^{th}$ species.
Ecologists often refer to $z_j$'s as the ``resources'' that the populations $x_i$ depend upon. The $c_i$'s are constant per capita growth or death rates.

{\bf Species on an island.}
While Equations of the form \eqref{eq:Z1} might occur in many circumstances, most simply $X(t)$ can be viewed as a collection of species on an island, and $Z(t)$ represents resources that ocean waves bring to the island--to be consumed by the $X$ species. The $Z$ species might also include migrating flying predators that seasonally attack some of the $X$ species.

{\black  \cite{smale1976differential} showed that Lotka–Volterra systems that have five or more species could exhibit any asymptotic behavior, including a fixed point, a limit cycle, an n-torus, or attractors, and our system is even more general.}

Our main result says that if $X(t)$ is a bounded solution of Eq.~\eqref{eq:Z} and the dimension of the kernel of $S^T$ is $k$ where $k>0$, then $k$ ``species'' or coordinates must die out exponentially fast. 

Suppose we had what might appear to be a more general form of Eq.~\eqref{eq:Z} where we replace $z_j(t)$ by $\hat z_j (X(t), X(t-\tau_j), t)$, {\black where $\tau_j$ is a time delay}. Suppose also $X(t)$ was a solution of this new equation. Then $X(t)$ would also be a solution of the original Eq.~\eqref{eq:Z} after we set $z_j(t) = \hat z_j(X(t), X(t-\tau_j),t)$ using the given $X(t)$.  

{\black For a given solution $X(t)$, }we say {\BF \bf $k$ ``species'' (or coordinates of $X(t)$)
die out exponentially fast} if there are constants $a$ and $b$ ($b>0$) such that for each $t \ge 0$ there are at least $k$ choices of $i$ for which $x_i(t) \le e^{a-bt}$. Which $x_i$'s are small may depend upon $t$. We can imagine that there is a minimum threshold level of population density below which $x_i$ cannot go and later recover, below which $x_i(t)$ will go to $0$.

Let $k>0$ be the dimension of the kernel of $S^T$, the transpose of $S$.
Assume there is a solution $X(t) := (x_1, \ldots, x_d)(t)$ where all $x_i(t)$ are bounded for $t \ge 0$. 
Theorem~\ref{thm:DieOut1} addresses the following questions: 
\begin{enumerate}
\item Must some $x_i$  die out?   (Yes, since the kernel dimension $k$ is greater than zero.)
\item What is the minimum number of $x_i$ that must die out? (The answer is $k$ for almost every choice of $C$.)
\item Must the above $k$ $x_i$  die out exponentially fast? (The answer is almost certainly yes! Establishing that is the main objective of this paper.)
\item Can we tell which species must die out? (The answer depends upon $C$ and $S$, and sometimes is yes, sometimes no.)
\end{enumerate}

Our focus on the kernel dimension $k$ is motivated by  
\cite{jahedi2022global,YorkeSaureJahedi}, which investigates 
 nonlinear equations of the form $F(x) = C$ for $F:\R^N\to\R^M$, where the dimension of the kernel of $DF(x)$ is at least $k>0$ for all $x$. They 
use $k$ to determine 
nature of the sets of solutions. When these equations represent steady states of ecosystems, they argue that only $N-k$ populations can coexist with positive values in a steady state.

\begin{defn}\label{def dieout}
We will say {\BF\bf
Eq. \eqref{eq:Z} has a bounded solution $X$}
when $X(t)$
satisfies Eq. \eqref{eq:Z} 
for some function $Z(t)$ for all $t\ge0$, 
and is bounded with bound $\beta$, \ie,
 $x_i(t) \le \beta$ for all $i \in \{1, \ldots, d\}$ and $t\ge0$).

We say {\bf\BF $k>0$ species  (or coordinates of $X(t)$)} die out exponen\-tially fast if
there exist $a^*,b^*\in \R$ where $b^*>0$ such that for each $t$ there are at least k choices of $i$ for which
\begin{align}\label{ineq ab thm}
   x_i(t)\le e^{a^* - b^*t}.
\end{align}
\end{defn}

\bigskip
\noindent
Which $k$ species are small and satisfy \eqref{ineq ab thm} at time $t$ may vary with time so that no species $x_i$ satisfies $x_i(t) \to 0$ as $t \to \infty$; see Figure~\ref{fig:3 half planes}(D) of Example \ref{ex specific}. Heuristically, we can expect that eventually, those species must be so small that they die out and become 0.

\begin{customthm}{2}[{\bf\BF $k$ species die out}]\label{thm:DieOut1}
Assume $k:=$ dimension of the kernel of $S^T$ satisfies $k>0$.  
Then the following holds for almost every $C$:\\
if Eq.
\eqref{eq:Z} has a bounded solution $X$, 
then at least $k$ species (or coordinates of $X(t))$ die out and do so exponentially fast.
\end{customthm}
\black

The
Lotka-Volterra population model Eq.~\eqref{eq:LV} that occurs in many fields of science and engineering (\cite{gause1932experimental}, \cite{grover1997resource})
is a special case of Eq.~\eqref{eq:Z}.

That it is a special case  can be seen by substituting  variables $z_j$ 
for $x_j$ on the right-hand side of Eq.~\eqref{eq:Z}.
Then if $X(t)$ is a solution of Eq.~\eqref{eq:LV} 
and we set $z_j(t)=x_j(t)$ for all $j$, $X(t)$ will be a solution of Eq.~\eqref{eq:Z}, now with $d'=d$.

Eq.~\eqref {eq:Z},
allows us to focus on particular coordinates of the dynamics that are primarily responsible for some coordinates dying out, so its matrix
$\SZ$ is a submatrix of $\SLV$ in the more limited model, Eq.~\eqref{eq:LV}.

One criticism of our model could be the over simplicity of the linear interactions between species or coordinates as incorporated by the matrix $s_{ij}$. 
There is a large literature in which the interactions between species are modeled in more complex ways, \cite{dubey2004persistence}.
We view our paper as an investigation of the nature of die-out functions more than of ecology and as pure mathematics aimed at the applied scientist. By keeping interactions simple, we hope the model and the methods might be applicable to situations throughout the sciences.
We believe this technique of finding multiple die-out Lyapunov functions will have applications far beyond ecology. 

\begin{figure} 
 \begin{center}
\centering
\resizebox{12.5cm}{!}{%
 \begin{tikzpicture}
[xshift=-3cm,ultra thick,node distance=1cm]
  \node[bluenode](c1)[xshift=-8.9cm]{$11$};
    \node[bluenode](c2)[xshift=-7cm]{$10$};
      \node[bluenode](c3)[xshift=-5cm]{$9$};
        \node[bluenode](c4)[xshift=-3cm]{$8$};
   \node[bluenode](c5)[]{$7$};
  \node[bluenode](c6)[xshift=+2cm]{$6$};
   \node[bluenode](c7)[xshift=+4cm]{$5$};
    \node[bluenode](c8)[xshift=+6cm]{${4}$};
\node[whitenode](r0)[xshift=-0.5cm,yshift=+5cm]{$14$};
\node[rednode](r1)[xshift=+2cm,yshift=+4cm]{$13$};
\node[rednode](r2)[xshift=-2cm,yshift=3cm]{$12$};
\node[rednode](r3)[xshift=-5.4cm,yshift=-3cm]{${3}$};    
\node[rednode](r4)[xshift=0cm,yshift=-4cm]{${2}$};    
\node[rednode](r5)[xshift=5cm,yshift=-5cm]{${1}$};
\draw[<->,black] (r0)-- (r2);
\draw[<->,black] (r0)-- (r1);
\draw[<->,black] (r1)-- (r3);
\draw[<->,black] (r1)-- (r4);
\draw[<->,black] (r1)-- (r5);
\draw[<->,black] (r1)-- (c1);
\draw[<->,black] (r1)-- (c2);
\draw[<->,black] (r1)-- (c3);
\draw[<->,black] (r1)-- (c4);
\draw[<->,black] (r1)-- (c5);
\draw[<->,black] (r1)-- (c6);
\draw[<->,black] (r1)-- (c7);
\draw[<->,black] (r1)-- (c8);
\draw[<->,black] (r2)-- (r1);
\draw[<->,black] (r2)-- (r3);
\draw[<->,black] (r2)-- (r4);
\draw[<->,black] (r2)-- (r5);
\draw[<->,black] (r2)-- (c1);
\draw[<->,black] (r2)-- (c2);
\draw[<->,black] (r2)-- (c3);
\draw[<->,black] (r2)-- (c4);
\draw[<->,black] (r2)-- (c5);
\draw[<->,black] (r2)-- (c6);
\draw[<->,black] (r2)-- (c7);
\draw[<->,black] (r2)-- (c8);
\draw[<->,black] (r3)-- (r4);
\draw[<->,black] (r3)-- (c1);
\draw[<->,black] (r3)-- (c2);
\draw[<->,black] (r3)-- (c3);
\draw[<->,black] (r3)-- (c4);
\draw[<->,black] (r3)-- (c5);
\draw[<->,black] (r3)-- (c6);
\draw[<->,black] (r3)-- (c7);
\draw[<->,black] (r3)-- (c8);
\draw[<->,black] (r4)-- (r5);
\draw[<->,black] (r4)-- (c1);
\draw[<->,black] (r4)-- (c2);
\draw[<->,black] (r4)-- (c3);
\draw[<->,black] (r4)-- (c4);
\draw[<->,black] (r4)-- (c5);
\draw[<->,black] (r4)-- (c6);
\draw[<->,black] (r4)-- (c7);
\draw[<->,black] (r4)-- (c8);
\draw[<->,black] (r5)-- (c1);
\draw[<->,black] (r5)-- (c2);
\draw[<->,black] (r5)-- (c3);
\draw[<->,black] (r5)-- (c4);
\draw[<->,black] (r5)-- (c5);
\draw[<->,black] (r5)-- (c6);
\draw[<->,black] (r5)-- (c7);
\draw[<->,black] (r5)-- (c8);
\draw[->,black](r3)edge[in=-20,out=60,loop below]node[below right]{}();
\draw[->,black](r4)edge[in=-20,out=60,loop below]node[below right]{}();
\draw[->,black](r5)edge[in=-20,out=60,loop below]node[below right]{}();
\draw[->,black](r1)edge[in=-20,out=60,loop above]node[above right]{}();
\draw[->,black](r2)edge[in=-20,out=60,loop above]node[above right]{}();
\draw[->,black](r0)edge[in=-20,out=60,loop above]node[above right]{}();
                \node[](s)[xshift=-8cm,yshift=5cm,label={[xshift=-0.cm, yshift=0cm]{{\textbf{S}}}}]{};    
                \node[](s0)[xshift=-8cm,yshift=-5cm]{};
                
\draw[->,black] (s0)-- (s) node[xshift=-0.5cm, yshift=-0.2cm,label={[xshift=-0.1cm, yshift=-9.3cm]\LARGE1},label={[xshift=-0.1cm, yshift=-8.4cm]\LARGE2},label={[xshift=-0.1cm, yshift=-7.6cm]\LARGE3},label={[xshift=-0.1cm, yshift=-5.2cm]\LARGE4},label={[xshift=-0.1cm, yshift=-3cm]\LARGE5},label={[xshift=-0.1cm, yshift=-2cm]\LARGE6},label={[xshift=-0.1cm, yshift=-1cm]\LARGE7}]{};
\end{tikzpicture}
}%
 \caption{{\textbf{A system with $14$ species where three must die out.}}
  There are $8$ blue nodes,  numbers 4-11, that have connections only with the $5$ nodes (colored red), numbers 1,2,3,12, and 13.
  The numbers on the left indicate the trophic levels. Imagine that the higher levels species are eating the lower-level ones if there is an edge connecting them. Under certain conditions including the form of the equations, Theorem~\ref {thm:trap}  shows there must be a globally attracting trapping region, andTheorem~\ref{thm:DieOut1} shows that for almost every choice of coefficients in the corresponding Lotka-Volterra system, at least three of the blue-node species must die out simultaneously, exponentially fast. For more discussion, see Example~\ref{sec:example 3}.
}
 \label{fg:14-trophic}
\end{center}
\end{figure}
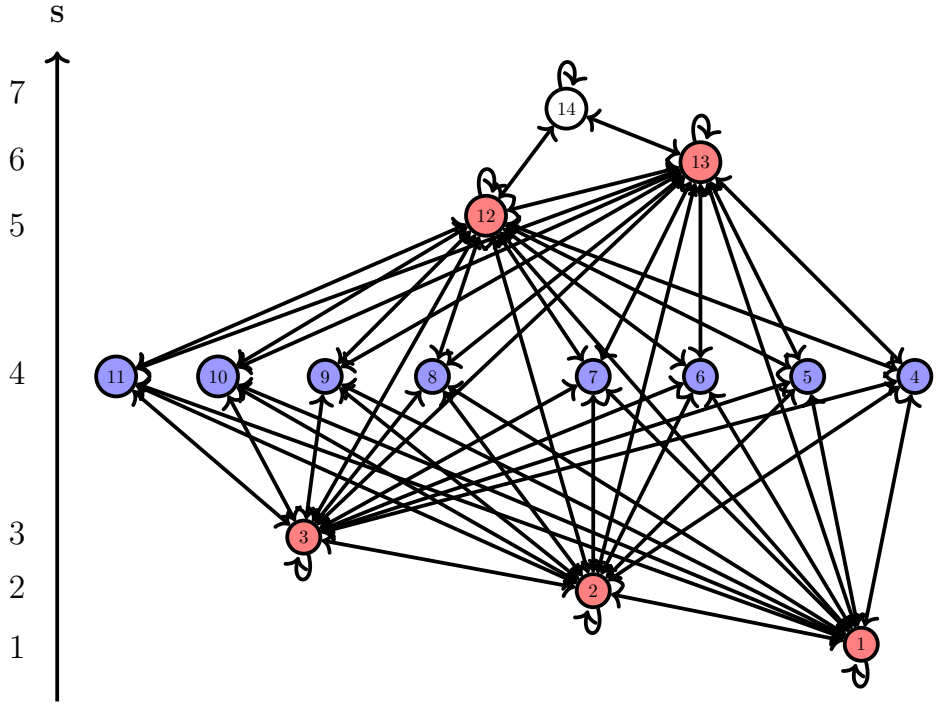

{\bf  Figure~\ref{fg:14-trophic} represents a 14-species ecosystem.} 
Ecologists might say they don't see systems like this in nature. The reason is given by Theorem \ref{thm:DieOut1}: It suggests that at least three species will die out and they will die out fast.
This figure will help make it clear what the above theorems are saying.
Each node represents a different species. We write $x_i$ for the population density of species $i$.
If there is a directed edge from species $j$ to species $i$, then the the logarithmic rate of change of species $i$ depends directly on species $j$; \ie, 
\beq
 x_i'(t) 
  = x_i(t)f_i(\ldots,x_j,\ldots),
\eeq
where $f_i$ depends on some of the species densities (including species $j$ since there is an edge from $j$ to $i$) but not necessarily all of the species densities. If there is no such edge, then $s_{ij} =0$.
For such a system, Eq.~\eqref{eq:LV1}
is 14 dimensional, $d = 14$. In contrast, Eq.~\eqref{eq:Z1} has dimension $d=8$ and $d'=5$. The vector $X$ includes only the 8 blue nodes and $Z$ includes only the 5 red nodes. 
Then the $5\times8$ matrix $S^T$ has a kernel of at least dimension 3. Hence by the Theorem~\ref{thm:DieOut1}, we can expect that 3 species must die out. Since the image of a subspace of dimension 8 has an image of dimension 5, the kernel must have at least dimension 3, and the same is true for the full 14-dimensional space.

{\bf A nonautonomous Eq.~\eqref{eq:Z1}}. If we assume that the system has the form  Eq.~\eqref{eq:LV}, 
then it can be modified in many ways so long as the eight equations of Eq.~\eqref{eq:Z1} are unchanged and there is still a bounded solution.
For example for $i$ that is not one of the blue nodes, the growth rate $c_i$ can be made a function of time. For example $c_i(t)$ could reflect the rainfall available to plants. Then the system would be non-autonomous. As long as the 8-dimensional vector $C$ and the $8$ by $5$ matrix is unchanged, the result does not depend on how $Z(t)$ is generated.

\section{Introduction to our generalized Lyapunov functions}

Usually, a Lyapunov function $V(X)$ is only used to establish the stability of a steady state, but that is not our approach.
We use the term {\bf Lyapunov function} for any real-valued function $V(X)$ for which $V(X(t))$ is monotonically decreasing in some specified regions.

For the $4$-dimensional Lotka-Volterra model in Fig.~\ref{fig:3preds1prey-dynamics}, {\black we will show that $k=2$ and that there are three die-out Lyapunov functions (see Eqs.~\eqref{eq:34}-\eqref{eq:32}). }

 Suppose Fig.~\ref{fg:4-troph c} represents Lotka-Volterra system for which all species are bounded for $t \ge 0$.
Then at least two species must die out exponentially fast. 
The species dying out are from species numbers $2,3$, and $4$. 

Fig.~\ref{fg:14-trophic} (Example \ref{sec:example 3}), is a more complicated example. It is sometimes necessary to look at more than $k$ Lyapunov functions. There, we examine $28$ Lyapunov functions to show three species must die out. 

Section~\ref{sec examples} provides several examples.
{\black Next we address the existence of a bounded solution. }

The ideas needed for the proof of Theorem~\ref{thm:DieOut1} are sufficiently novel that we have chosen to introduce these ideas here rather than to jump directly into the proof.

{\bf\BF Defining $\dt{V}$ for a real-valued function $V$.}
Let $U$ be an open set in $\R^d$.
Assume $F:U\to\R^d$ and $V: U \to \mathbb R$ are $C^1$.
For the differential equation
  \begin{align}\label{DE}
     \frac{dX}{dt} = F(X),
 \end{align}
 define 
\begin{align}\label{dotV}
\dt V(X) :=\nabla V(X) \cdot F(X) 
\text{ so that } \frac{d}{dt}V(X(t)) =  \dt V(X(t)).
\end{align}
Hence  $\dt{V}(X)$ tells how fast $V(X(t))$ changes when the trajectory is passing through a point $X$  -- without the user  knowing the actual solution $X(t)$. Of course if $F$ depends on $X$ and $t$, $\dt V(X,t) =\nabla V(X) \cdot F(X,t)$.

\begin{defn}\label{def Die-out}
Let $V:\cP \to\R$.
In this paper, we will say a real-valued function $V$ is a {\bf{die-out Lyapunov function}} if $\dt V(X) <0$ on $\cP$.
\end{defn}

\begin{defn}[The form of our die-out functions]\label{def:nuS}
A vector $\nu$ is a {\textbf{{null vector of the transpose matrix $S^T$}}} if $S^T \nu^T =0$, or equivalently $\nu S =0$. We define the function $\Lambda_{\nu} (X)$ on $\cP$, as follows
\begin{align}\label{lyap}
    \Lambda_{\nu} (X) &:= \nu\bigcdot \ln X,
\end{align}
\ie, the inner product of $\nu := (\nu_1, \ldots, \nu_d)$ with $\ln X := (\ln x_1, \ldots, \ln x_d)$. 
\end{defn}
From \eqref{eq:Z1}, it has the property that  
\beq\label{nu dot C}
\dt \Lambda_\nu(X)=\nu\cdot (\frac{x_1'}{x_1},\ldots,\frac{x_d'}{x_d})= \nu\cdot (C+SZ)
= \nu\cdot C+\nu SZ = \nu\cdot C
\eeq
since $\nu S = 0$. 
Hence $\dt \Lambda_\nu(X)$ is constant for Eq.~\eqref{eq:Z} --- and for its special case Eq.~\eqref{eq:LV}. If that constant is negative, then it is a die-out function. 

\begin{example}\normalfont
Suppose three species depend on two resources, and those three only interact with the two resources. We might refer to their population densities as $x_1,x_2, x_3$ and $z_1,z_2$, respectively. They might be part of a much bigger ecosystem, but we can focus on this subset. Suppose further that $S$, restricted to the domain $x_1,x_2, x_3$, has the form:
$$S := \left[
\begin{array}{cc}
      1 & 2 \\
      1 & 1 \\
      3 & 1 
    \end{array}\right].$$
Its kernel is one-dimensional, and it has a null vector $\nu:= \left[2,-5,1 \right]$ of $S^T$. Then 
\begin{align}\label{intro null vec}
\Lambda_\nu(X) := 2\ln x_1 -5 \ln x_2 +\ln x_3 \text{ \ on }  \cP,
\end{align}
(which appears in Example~\ref{ex specific})
and assume (as will be the case in this paper) that $C := \left[
\begin{array}{c}
      -1  \\
      1  \\
      -1  
    \end{array}\right]$. Then $\nu\cdot C <0$. In particular,
\begin{align*}
\dt \Lambda_\nu(X) = \nu \bigcdot C = -8 \ne0 \text{ \ on }  \cP.
\end{align*}
Then for any 
bounded solution $X(t)$  of Eq.~\eqref {eq:Z}, either species/coordinates 1 or 3 must  be dying out at each $t$ because {\bf\BF these are the coordinates for which $\nu_i>0$}. Other species might also die out as well.

This example is a special case of Example \ref{ex specific} in Sec.~\ref{sec examples}.  There we show that while $\min \{x_1(t), x_3(t)\} \to 0$, it is possible to choose $Z(t)$ such that neither $x_1$ nor $x_3$ goes to $0$ as $t \to \infty$. 

Now follow some technical details as to why ``$\min \to 0$".
Since $X(t)$ is bounded by $\beta$, \ie, $x_i(t) \le \beta$ for all $t>0$ and {\black all $i$. Since} the above constant $\nu \bigcdot C$ is negative,  $\Lambda_\nu(X(t))\to -\infty$ as $t\to\infty$. The term $-5\ln x_2$ is bounded from below, since $-5\ln x_2 \ge -5 \ln \beta$. Hence, $2\ln x_1 + \ln x_3 \to -\infty$, so $\min \{\ln x_1, \ln x_3\} \to -\infty$, and equivalently,
\beq\label{must die}
\min\{x_1(t),x_3(t)\}\to 0  \text{ as } t\to\infty.
\eeq
For a different $C$, we can have $\nu\bigcdot C >0$, in which case, then $\Lambda_\nu(X(t))\to+\infty$ as $t\to\infty$ which implies
$$x_2(t)\to 0  \text{ as } t\to\infty.$$
\end{example}
{\bf Teams of die-out Lyapunov functions. }
The above argument can be used to produce a die-out function whenever $m+1$ species depend on $m$ (for almost every choice of $C$ and $S$). 

Now suppose there is an ecosystem with 14 species as in Fig~\ref{fg:14-trophic}. Each node represents a species. There is a connection from node $j$ to node $i$ if $\frac{x_i'}{x_i}$ is a function of $x_j$. There is a group of 5 species (colored red) and a group of 8 species (colored blue), each of which interacts directly only with the 5 reds. That is, each edge from a blue node connects only with red nodes. The red nodes have no limitations as to what they connect with. Choose any 6 of the 8.
Then these 6 depend only on the 5, and we can construct a die-out Lyapunov function that depends only on those 6 $x$ variables. We conclude that at least one must die out (exponentially fast). But there are 28 ways to choose 6 species out of 8. Hence there are 28 die-out Lyapunov functions. This seems like a complicated mess. Our proof of Theorem~\ref{thm:DieOut1} shows that at least 3 species ($8-5$) must die out exponentially fast. 

In this paper, each example may have several such functions $V$, each depending on a different subset of $x_i$. Each makes a different statement about which species must die out. To understand the entire picture, many or all of these die-out Lyapunov functions must be considered together.
We call the collection of these a ``team of Lyapunov functions''.
While each $V$ guarantees some coordinate(s) $x_i$ are going to 0 in some sense, the team will allow us to conclude that several are going to 0 in some sense, simultaneously. 
Our emphasis is on detecting how many are going to 0.

{\bf\BF When $\nu \bigcdot C =0$.} When several species such as rabbits, deer, and cows are feeding on the same resources such as grasses and weeds, they can coexist if their per capita growth rates are ``perfectly balanced'' with $S$. That condition is that $\nu \bigcdot C =0$ for each null vector $\nu$ (see corollary~\ref{cor1}).

The Discussion section, Sec.~\ref{sec discussion}, provides additional insights into the themes of the paper.
\begin{figure}[ht]
\begin{center}
\centering
 \begin{tikzpicture}
[xshift=-3cm,ultra thick,node distance=1cm]
  \node[bluenode](c1)[xshift=-2cm]{$x_2$};
  \node[bluenode](c2)[]{$x_3$};
\node[bluenode](c3)[xshift=+2cm]{$x_4$};
\node[rednode](r1)[xshift=0cm,yshift=-2cm]{$x_1$};
\draw[<->,black] (r1)-- (c1);
\draw[<->,black] (r1)-- (c3);
\draw[<->,black] (r1)-- (c2);
\draw[->,black](r1)edge[in=-20,out=60,loop below]node[below right]{}();

   \node[]at(5,-.8){ $S =\left[
\begin{array}{ccccc}
      s_{11} & s_{12} & s_{13} & s_{14} \\
      s_{21} & 0 & 0 & 0 \\
      s_{31} & 0 & 0 & 0 \\
      s_{41} & 0 & 0 & 0
    \end{array}\right]$};
    \end{tikzpicture}
\end{center}
 \caption{{\bf Three predators and one prey.} See Example~\ref{ex:4D} {\black below in Section \ref{sec examples} for more details}.
In our figures, each node represents the population density of a species.
The existence of an edge from node $j$ to node $i$ means species $j$ directly influences species $i$, \ie, $s_{ij}\ne 0$. 
Some species $i$ here may have ``self-influence'' meaning $s_{ii}\ne 0$, indicated by an edge going from its node back to its node.
An edge between two nodes with an arrow at each end means each of the two directly influences the other.
The graph implies $S$ has the form shown. 
 }
\label{fg:4-troph c}
 \end{figure}
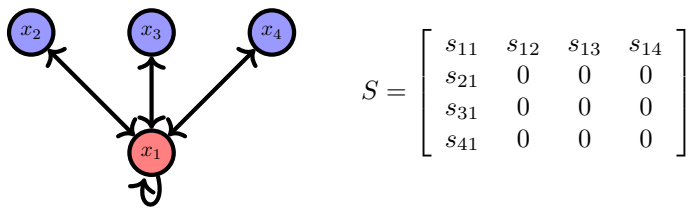
\begin{figure}[ht]
\begin{subfigure}{.5\linewidth}
\centering
\includegraphics[width=1.\linewidth]{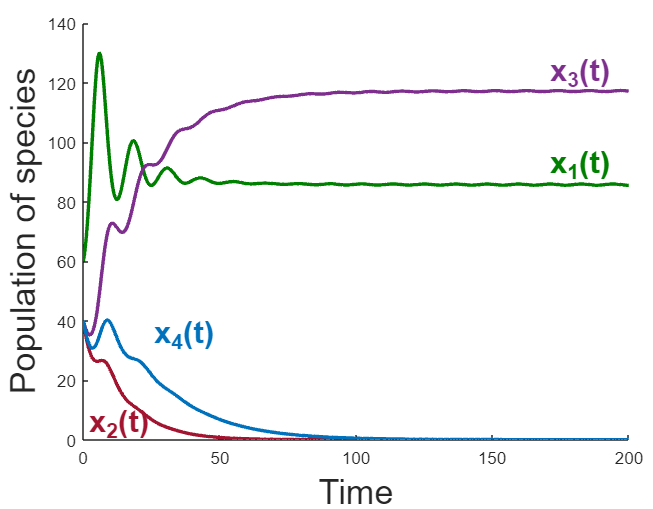}
\caption{}
\end{subfigure}
\begin{subfigure}{.49\linewidth}
\centering
\includegraphics[width=1.\linewidth]{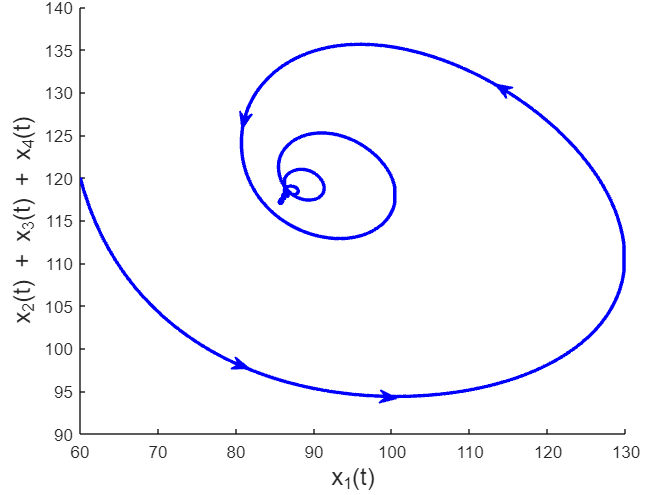}
\caption{}
\end{subfigure}\\[1ex]
\caption{
{Populations of three predators and one prey are plotted. See Example~\ref{ex:4D}.
 Parameters are given in table \eqref{table:Symbol2}.
}
In this case, only one predator and the prey survive and tend to the steady state. 
Two predators die out.
In this case, which survives can sometimes be determined from the null vectors of $S^T$, as we shall show. 
}
 \label{fig:3preds1prey-dynamics}
\end{figure}
\section{Background}
\label{sec:background}

See Sec.~\ref{sec:background} for a discussion of the origins of our methods.

This paper is motivated by connecting the ideas of two papers, \cite{jahedi2022global,jahedi2022robustness}, which apply null vectors in their generalized competitive exclusion principle.
It is also motivated by \cite{akhavan2020population} in which a Lyapunov-like function of the form $V := \frac{x_1}{x_2}$ was used to show one species must die out in a population model. 

\cite{mcgehee1977some} and   \cite{armstrong1980competitive} investigate Lotka-Volterra 
systems Eq.~\eqref{eq:Z} 
and more general autonomous systems for which they show, the answer to question (1) is yes. They show that the limit set of a trajectory $X(t)$ as $t\to\infty$ cannot contain any points $X_0\in\cP$.
Hence if $X_0$ is a limit point of $X(t)$, some coordinate of $X_0$ must be zero, {\ie, $X_0 \notin \mathcal P$.}

If we add  to the results of \cite{mcgehee1977some} the assumption that there is a bounded trajectory, then their results imply that at least one specie dies out ---
in the sense that $\min_i x_i(t) \to0$ as $t\to\infty$. As time passes, at each time $t$, for at least one coordinate $i$, $x_i(t)$ is small, but which species is small may vary from time to time. 
Nonetheless, one can argue that in practice, eventually at least, one specie density is so small that it cannot recover and it dies out.
{\black They conclude only that {\black eventually,} at least one species dies out and provide no information about the speed of the die-off.
Our Props.~\ref{prop1} and \ref{prop V double dot} have a similar spirit and are included here in part to relate our results to theirs.}

To give the reader the spirit of our methods in attacking the above questions, we present the following result. 
Notice that this result makes no assumptions about what values $V$ assumes. $V$ does not have to take on only non-negative values.

 \begin{prop}\label{prop1} 
Assume there exists a $C^1$ differential equation \eqref{DE} where $F$ is defined on
an open set $U$ and $F$ is $C^1$. Assume \\
($A_1$) there is a trajectory $X(t):=(x_1,\ldots,x_d)(t)\in U$ defined for all $t\ge 0$, and\\
($A_2$) $V:U\to\R$ is differentiable and
$\dt V(X) < 0$ for each $X\in U$.\\
Then\\
($B_1$) $X(t)$ has no limit points in $U$ as $t\to\infty$.\\
($B_2$)
 If furthermore $U=\cP := \{X:\text{for all }i, x_i>0\}$ and the trajectory $X(t)$ is bounded (\ie, there is a $\beta>0$ for which $x_i(t)\le\beta$ for all $i$ and $t\ge0$),
then 
\begin{align}\label{prop1 to 0}
    \displaystyle{\min_{1\le i\le d} x_i(t)}\to 0\text{ as }t\to\infty.
\end{align}
\end{prop}
When {\it($A_2$)} is satisfied,
we will call $V$ either a ``die-out'' Lyapunov function or, more simply, a Lyapunov function. 

We have made no assumptions about the sign of $V$ nor whether $V$ is bounded nor whether it is positive definite. In ($A_2$) we could alternatively assume $\dt V(X)>0$ on $U$ since we need only assume that $\dt V(X)$ is never zero. 
The result does not require $U$ to be invariant or positively invariant since it only makes statements about those trajectories that stay in $U$ for all future times, if there are any.
Conclusion ($B_1$) is true since $\dt V$ would be 0 at any such limit point, which contradicts ($A_2$); and ($B_2$) follows from ($B_1$)  since if \eqref{prop1 to 0} is false, there exists a sequence $t_n \to \infty$ and an $\varepsilon>0$ such that $\displaystyle{\min_{1 \le i \le d} x_i(t_n)} \ge \varepsilon$. {\black Since $X(t)$ is bounded, the sequence $\{X(t_n)\}$ has a limit point $X^*$ in $\cP$ as $n \to \infty$,} which contradicts ($B_1$). See other limit point methods in 
Barbashin-Krasovskii-LaSalle \cite[p.~309]{alligood1996chaos}.

See Prop.~\ref{prop V double dot} (Sec.~\ref{sec discussion}) for a significant generalization of Prop.~\ref{prop1}, though there, the domain $U$ is required to be simply connected.


\black
\section{{Die-out Theorem}}\label{sec:lyapunov}
The proof of Theorem~\ref{thm:DieOut1} is at the end of this section following our key Lemma~\ref{lemma:die-out}.

We will say a solution $X$ is {\bf doubly bounded} if for some $\beta >0$, 
$\frac{1}{\beta} < x_j(t)<\beta$ for all $t\ge0$ and all $1\le j\le d$.

\begin{cor}\label{cor1}
If for some $C$ there is a  doubly bounded solution $X$ of Eq. \eqref{eq:Z}, then  $\nu \bigcdot C = 0$ for every null vector $\nu$ {\black of $S^T$}.
\end{cor}

\begin{proof}
If $X$ is a bounded solution and $\nu$ is a null vector such that $\nu \bigcdot C$ is non-zero, then some species must die out so $X$ cannot be doubly bounded. But $X$ is doubly bounded, so $\nu \bigcdot C =0$ for every null vector $\nu$.
\end{proof}

We present several ``Facts'' and definitions  needed for the proof. These facts are mini propositions that present new core concepts and ideas. They represent original ideas that merit being highlighted as  central to the proofs, but their proofs are too simple for us to label them with the august title of Proposition.


Given a bounded solution $X(t)$, Theorem~\ref{thm:DieOut1} says that for each $t$ there are at least $k$ choices of $i$ for which Ineq.~\eqref{ineq ab thm} is satisfied. 
We can get constraints as to which $x_i$'s are small at some times $t$ from the study of the null vectors.

\black
To clarify ``almost every $C$'', consider the trivial case in which the matrix $S$ is identically zero. Then the differential equation is very simple:
$  x_i' = c_i x_i $ where $ i=1,\ldots,d.$
{\black Then there is a bounded solution only if every $c_i$ in $C$ satisfies $c_i \le 0$.
For almost every such $C$, $c_i<0$ for all $i$.}
For almost every {\black such} $C$, no coordinate of $C$ is zero. Hence for almost every $C$, either there is no bounded solution (when some $c_i$ is positive) or all solutions are bounded and all $x_i(t)$ go to zero as $t \to \infty$ (when $c_i<0$ for all $i$). Hence we may summarize the trivial system's behavior  by saying ``for almost every $C$, if there is a bounded solution, the solution goes to 0''.

{\bf Die-out Lyapunov functions.}
Consider the differential equations \eqref{eq:Z} on $\cP$.
The system of equations \eqref{eq:Z} can be rewritten on $\cP$ as
\begin{align}\label{equationX*-ndim}  
\left[
\begin{array}{c}
       \frac{x_1'}{x_1}\\
      \vdots \\
      \frac{x_d'}{x_d}
    \end{array}\right](t) &= C + S Z(t),
\end{align}
{\black where $Z(t)$ can have the form $Z(X(t),t)$.}

{\bf What is a ``Fact''?} The above theorem is a consequence of several key ideas that will guide the reader toward an overview of the proof.  The only originality with these facts is in their formulation, not in finding their concise proofs. We call these mini propositions ``Facts'' so as not to glorify them with full proposition status.  And we point out where they are needed.

\noindent
If $\eta$ is a null vector for which  $\eta\bigcdot C \ne 0$, then either $\nu:=\eta$ or $\nu:=-\eta$  satisfies 
$\nu\bigcdot C < 0$. 

This fact is used in the proof of Lemma \ref{lemma:die-out}. An equivalent statement is that there is some $j$ for which $\nu_j >0.$

\begin{Fact}
[$\dt\Lambda_\nu$ is constant]
\label{Lambda dot}
Assume some null vector $\nu$  (of $S^T$) satisfies $\nu\bigcdot C < 0$ for some $C$. 
Then for all $X\in\cP$,
\begin{align}\label{Lambda}
  \dt \Lambda_{\nu}(X) = \nu\bigcdot C<0,
\end{align}
\end{Fact}

This result is derived in Eq.~\eqref{nu dot C} above. \qed

\bigskip

If $\nu\bigcdot C \neq 0$, then either $\Lambda_\nu$ or $\Lambda_{-\nu}$ is a die-out Lyapunov function. {\black This is a rigorous, stronger version of the competitive exclusion principle (CEP) which asserts} that two or more predators cannot coexist if they are only limited by one prey species. 
The mathematical version says they can not coexist unless they are perfectly balanced in some sense, a situation that would only occur with probability zero. 
In the theorem below, we assume ``$\nu\bigcdot C \neq 0$'', which is precisely the statement that they are not balanced.

Let $\nu$ be a null vector of $S^T$. For almost every $C$, $\nu \bigcdot C \neq 0$. Assume $\nu \bigcdot C \neq 0$.  We can assume $\nu \bigcdot C<0$ since if it was $>0$, we could replace $\nu$ with the null vector $-\nu$. 

\subsection{A key Lemma}
The following Lemma will be useful in determining which coordinates must die out, which is crucial for proving Theorem~\ref{thm:DieOut1}.

We need a preliminary fact to determine the exponential rate of decay for a specified $\nu$. Define
\begin{align}\label{nu plus}
\nu^+ :=  \sum_{j:\nu_j>0} \nu_j.
\end{align}
\begin{Fact}[\bf\BF $\nu^+$ Positivity]\label{fact positivity}
Assume Eq.~\eqref{eq:Z} has a bounded solution $X$, and some null vector $\nu$ satisfies $\nu\bigcdot C < 0$ for some $C$.  Then  $\nu^+\ne0$ (and it is positive).
\end{Fact}
This fact is needed to obtain the formula (Eq.~\eqref{def a b}) for $a_\nu$ and $b_{\nu}$ in the statement of Lemma \ref{lemma:die-out}.
\begin{proof}
This fact is equivalent to ``$\{j:\nu_j>0\}$ is non empty'' in Eq.~\eqref{nu plus}. We need to show 
$\nu_i$ is $ > 0$ for some $i$.
The hypotheses imply that $\Lambda(X)\to\infty$ as $t\to \infty$ so there is an $i$ and a sequence $t_n\to\infty$ such that $\nu_i \ln(x_i(t_n))\to -\infty$ as $n\to\infty$. Since $\ln(x_i)\le \ln(\beta)$, the only way for $\nu_i \ln(x_i(t_n))\to -\infty$ to occur is for 
$ \ln(x_i(t_n))\to -\infty$ which implies $\nu_i>0.$
\end{proof}

The null vector $\frac{\nu}{\nu^+}$ will play an important role. Fact~\ref{fact positivity} guarantees we are not dividing by 0 in Eqs.~\eqref{def a b} since $\nu^+>0$. Recall Eq.~\eqref{lyap} defines $\Lambda_{\nu} (X) := \nu\bigcdot \ln X$.

\begin{lemma}[\textbf{``Die-out'' Lyapunov Functions}]\label{lemma:die-out}
 Assume the hypotheses of Fact \ref{fact positivity}, \ie, Eq.~
\eqref{eq:Z} has a bounded solution $X$ and  some null vector $\nu$  satisfies $\nu\bigcdot C < 0$ for some $C$.\\
Then there exist $a_{\nu},b_{\nu}\in\R, b>0$, such that 
\begin{align}\label{ineq ab lem}
\min_{i: {\nu_i>0}} x_i(t)\le e^{a_{\nu}-b_{\nu}t}\text{ for all } t \ge 0.
\end{align}
where
\begin{align}\label{def a b}
    a_\nu := -\ln \beta \dis\sum_{i: \nu_i <0} \frac{\nu_i}{\displaystyle{\nu^+}} + {\displaystyle\sum_{i}} \frac{\nu_i}{\displaystyle{\nu^+} }\ln x_i(0), \text{ and }
    b_\nu := |\frac{\nu }{\displaystyle{\nu^+}}\bigcdot C|.
\end{align}
Furthermore $b_{\nu}$  and $a_{\nu}$
depend only on $\frac{\nu}{\nu^+}$ and on $C$ or $X(0)$ and $\beta$, respectively.
\end{lemma}

{\bf Which species must die out?}
    This Lemma concludes that some species must die out exponentially fast. In Ineq.~\eqref{ineq ab lem},
    the min is computed only over those $i$ for which $\nu_i>0$.
    If $\nu\cdot C <0$, there is some $i$ for which $\nu_i>0$ by Fact  \ref{fact positivity}. As mentioned in Eq. \eqref{must die}, species 1 or 3 must die out  in the sense that $\min \{x_1(t), x_3(t)\} \to 0$ exponentially fast.

\begin{proof}
From the existence of 
 some null vector $\nu$ that satisfies $\nu\bigcdot C < 0$
and Fact~\ref{Lambda dot} and Eq.~\eqref{Lambda},

 \begin{equation}\label{Eq:3-in-proof}
    \dt{\Lambda}_{\nu}=  \nu\bigcdot C <0.
 \end{equation}
Hence, by definition, $\Lambda_{\nu}$ is a die-out Lyapunov function. 

Let $X(t)$ be a solution that, as in the statement of the proposition, is bounded with bound $\beta$, \ie, there exists $\beta>0$ such that  $x_i(t) \le \beta$ for each $1\le i\le d$. 
From Eq.~\eqref{Eq:3-in-proof}, 
\begin{align}\label{Eq:4-in-proof}
    \Lambda_{\nu}(X(t)) - \Lambda_{\nu}(X(0)) = (\nu\bigcdot C)t 
\end{align}
For each $t\ge 0$, let
\begin{align*}
    m_\nu(t) := \dis{\min_{i: {\nu_i>0}} x_i(t)}.
\end{align*}
Then , $m_{\nu} \le x_i \le \beta$ implies $\ln m_{\nu} \le \ln x_i \le \ln \beta$
since $\ln$ is a monotonic increasing function.
Multiplying by $\nu^+$ and summing over $i$ for which $\nu_i>0$ and using \eqref{nu plus}, gives
\begin{align*}
    \nu^+ \ln m_{\nu} 
= \sum_{i: \nu_i > 0} \nu_i \ln m_{\nu}  
    \le \sum_{i: \nu_i > 0} \nu_i \ln x_i, 
\end{align*}    
For $\nu_i <0$, $\nu_i \ln x_i \ge \nu_i \ln \beta$, so summing over $i$ for which $\nu_i <0$ yields
\begin{align*}    
    \ln \beta \sum_{i: \nu_i <0} \nu_i &\le \sum_{i:\nu_i <0} \nu_i \ln x_i;
\end{align*}
\black
therefore, adding ``$\nu_i >0$'' and ``$\nu_i <0$'' terms, 
\begin{align*}
  \begin{split}
    \nu^+ \ln m_{\nu} + \ln \beta \sum_{i: \nu_i <0} \nu_i 
    &\le \sum_{i} \nu_i \ln x_i =\Lambda_\nu (X(t)) \\
    &= \Lambda_\nu (X(0)) - |\nu \bigcdot C| t.
    \end{split}
 \end{align*}
Hence solving for $\ln m_{\nu}$ gives
\begin{align*}
    \ln m_{\nu}(t) \le a_\nu - b_\nu t,
\end{align*}
where
\begin{align*}
    a_\nu := \frac{-\ln \beta {\displaystyle\sum_{i: \nu_i <0}} \nu_i + {\displaystyle\sum_{i}} \nu_i \ln x_i(0)}{\displaystyle{\nu^+}}, \text{ and }
    b_{\nu} := \frac{|\nu \bigcdot C|}{\displaystyle{\nu^+}}.
\end{align*}
which is equivalent to Eqs.~\eqref{def a b}.
Hence Eqs.~\eqref{ineq ab lem} is satisfied. 

Therefore we can write $a_{\nu}$ and $b_{\nu}$ as functions of a normalized $\nu$ such as $\frac{\nu}{\nu^+}$. Of course $a_{\nu}$ and $b_{\nu}$ are also depend upon $X(0), \beta$, and $C$.
\end{proof}
\black
\subsection{A team of Lyapunov functions}\label{team}

A ``team'' is a group working together toward a goal. 
We adopt the word {\black ``team''} for a group of Lyapunov functions working together.

{\bf\BF The team $\cT$ of null vectors for $S$.}
For simplicity, we will always write null vectors $\nu$  of $S^T$ as row vectors despite the fact that 
they are column vectors since 
we often think of $\nu$ as a left null vector of $S$, so 
$\nu S =0$. Then $\nu$ is a row vector.

\begin{defn}\label{support}
For a null vector $ \nu = (\nu_1, \ldots, \nu_d)$, define the {\bf\BF support of $\nu$}, 
 \begin{align*}
    \Supp(\nu) := \{i: \nu_i \neq 0\}.
\end{align*}
We say a null vector $\nu\ne0$ is a {\bf minimal-support} vector if there is no  null vector with strictly smaller support, \ie, there is no null vector $\eta$ where $\Supp(\eta)$ is a proper subset of 
$\Supp(\nu)$.

Let $\cT$ denote the set of minimal-support null vectors of $S^T$. We call $\cT$ the {\bf team} (of null vectors) of $S^T$. 
\end{defn}
\begin{Fact}[Almost every $C\in\R^d$]\label{fact T is lines}
The set $\cT\cup\{0\}$
consists of a finite number of lines that pass through the origin, one line for each distinct support set $\Supp(\nu)$ for $\nu\in\cT$.\\
Hence, for almost every $C\in\R^d$, $\nu\bigcdot C\ne0$ for all $\nu\in\cT$.
\end{Fact}
This Fact is used at the beginning of the Proof of Theorem~\ref{thm:DieOut1}.

The simplest matrix is where $S$ is $0$, so every vector is a null vector. {\black We will discuss this case in Example~\ref{ex specific}.}
For every vector $C$ there is a null vector $\nu$ for which 
$\nu\bigcdot C=0$. However, $\cT$ is much smaller than the set of all null vectors {\black when the null space has dimension greater than $1$.}
For the matrix $S=0$, the minimal support null vectors have one coordinate that is non 0. Hence $\cT\cup\{0\}$ consists of $d$ coordinate axes, and 
$\{C\in\R^d: \nu\bigcdot C\ne0 \text{ for all }\nu\in\cT\}$ is the set of vectors $C$ for which no coordinate is 0. Fact~\ref{fact T is lines} is a generalization of this case.

\begin{proof}
First we show if two minimal-support vectors $\nu$ and $\eta$ are linearly independent, then
\begin{align*}
    \Supp(\nu) \ne \Supp(\eta);
    \end{align*}
If $\nu$ is a  minimal-support vector, so is $\alpha\nu$ for each non-zero number $\alpha$. Hence $\cT \cup \{0\}$ is a finite set of lines through the origin.

If two linearly independent null vectors have the same support $\sigma$ {\black (see Def.~\ref{support})}, then there are null vectors with strictly smaller supports.
Taking linear combinations of the linearly independent null vectors,  for each coordinate $j\in\sigma$ we can find one or more null vectors with strictly smaller support, one whose support does not include $j$.

If $\nu\bigcdot C=0$, then $C$ is in the subspace that is perpendicular to one of the lines in $\cT$. Since there are finitely many lines, the set of such $C$ has measure 0.
\end{proof}

Of course, 0 is a null vector but its support is $\{1,2,\ldots,d\}$.
For each non-zero null vector $\nu$, the Lyapunov functions $\Lambda_\nu$ tells us some $x_i$ must die out where $i\in \Supp(\nu)$, (assuming $\nu\bigcdot C \ne 0$). Therefore the set or ``team'' of null vectors having the smallest or ``minimal'' support is most useful, and sometimes the team has many members.

It follows that for each non-zero null vector $\nu$, there are minimal-support null vectors, the union of whose supports is 
$\Supp(\nu)$. Furthermore, the original null vector $\nu$ is a linear combination of the minimal-support null vectors.

\begin{defn}
Let $k>0$ be the dimension of the kernel of $S^T$. 
We say a coordinate $j$ is a  {\bf ``kernel coordinate''} if there is some null vector $\nu$
whose $j^{th}$ coordinate value $\nu_j$ is non-zero.  This concept first appeared in 
\cite{jahedi2022global} and \cite{jahedi2022robustness}.

Let $\cK$ be the set of {\bf ``kernel coordinates''} $i$, coordinates for which $\nu_i>0$ for some non-zero null vector $\nu$ of $S^T$.
\end{defn}

Lemma~\ref{lemma:die-out} says there are $a_\nu$ and $b_\nu$ that tell how fast $\min_{i: {\nu_i>0}} x_i(t)$
dies out. Now we show there are alternative $a$ and $b$ that work for all $\nu\in\cT$, provided $\nu\bigcdot{C}<0$.
Because $a_\nu$ and $b_\nu$ are only defined when $\nu\bigcdot{C}<0$, define
\begin{align}
 \cT_C :=&\{\nu\in\cT:\nu\bigcdot{C}<0\} ,\\
 a(\cT):=& \sup\{a_\nu: \nu\in\cT_C\}, \text{ and }\label{a cT}\\
 b(\cT):=& \inf\{b_\nu: \nu\in\cT_C\};\label{b cT}
 \end{align}
see Ineq.~\eqref{ineq ab lem} in Lemma~\ref{lemma:die-out} for $a_\nu$ and $b_\nu$.
\begin{Fact}[There exist $a$,$b$ that provide a lower bound die out rate for all  $\nu\in\cT$]\label{fact max a min b}
Assume the dimension of the kernel of $S^T$ is $k>0$ and \ Eq.~\eqref{eq:Z} has a bounded solution $X$.
Assume $\nu\bigcdot{C}\ne 0$ for all $\nu\in\cT$.
Then $a(\cT)<\infty$ and $b(\cT)>0$, and for $\nu\in\cT_C$,
\begin{align}\label{ineq ab lem2}
\min_{i: {\nu_i>0}} x_i(t)\le e^{a(\cT)-b(\cT)t}\text{ for all } t \ge 0.
\end{align}
Notice that in addition to $\cT$,  $a(\cT)$ also depends on $\beta$ and  $X(0)$, and $b(\cT)$ depends upon $C$.
\end{Fact}
Fact \ref{fact max a min b} is used in the Proof of Theorem~\ref{thm:DieOut1}
to provide an exponential decay rate for $\nu\in\cT$. 
The proof of Fact~\ref{fact max a min b} shows that while there are infinitely many $\nu\in\cT$, there are only finitely many values of $a_\nu$ and $b_\nu$ for $\nu\in\cT$
(see Eqs.~\eqref{a cT} and~\eqref{b cT}).

\begin{proof}
Assume $\nu$ and $\eta$ are minimal-support vectors in $\cT_C$. Assume either $a_\nu\ne a_\eta$ or $b_\nu\ne b_\eta$. Then  
$\frac{\nu}{\nu^+}\ne\frac{\eta}{\eta^+}$ from Eqs.~\eqref{def a b}. Hence 
$\nu$ and $\eta$ are linearly independent.
Then $\Supp(\nu)\ne \Supp(\eta)$, since
if they had the same support $J$, there would be a non-zero linear combination $\psi$ of $\nu$ and $\eta$ for which $\psi_i=0$ for some $i\in J$.
Hence neither $\nu$ and $\eta$ would be minimal-support vectors.
Since there are only finitely many subsets of $1,\ldots,d$, there are only finitely many distinct values of $a_\nu$ and $b_\nu$ in Eqs.~\eqref{a cT} and \eqref{b cT}. Hence $a(\cT)<\infty$ and $b(\cT)>0$ since each $b_\nu$ is $>0$ for $\nu\in\cT_C$. 

Therefore for each $\nu\in\cT_C,$
\begin{align*}
    e^{a_{\nu}-b_{\nu}t}\le e^{a(\cT)-b(\cT)t} \text{ for all } t \ge 0.
\end{align*}
 Therefore, Ineq.~\eqref{ineq ab lem2} follows from Ineq.~\eqref{ineq ab lem}.
\end{proof}

We say species $i$ is {\bf \BF dying out at time $t$} if $x_i(t) \le e^{a(\cT) - b(\cT)t}$. Let $J$ be a set of $j$ coordinate numbers such as $\{1,3,5\}$ when $j=3$. 
Suppose we know that the $j$ species (or coordinates) in $J$ are ``dying out at time $t$'', and that some species in $\Supp (\nu)$ is ``dying out at time $t$''. 
When $J$ and $\Supp(\nu)$ are assumed to be disjoint, we know there are at least $j+1$ species that are dying out at time $t$. This knowledge is what Fact \ref{fact T} gives us.

\begin{Fact}\label{fact T}
Assume the dimension of the kernel of $S^T$ is $k>0$.
Let $0\le j< k$.
For each set $J\subset\{1,\dots,d\}$ of $j$ coordinates, there is a non-zero null vector $\nu\in \cT$ such that $J\cap\Supp(\nu)$ is empty.
\end{Fact}
This Fact is used in the Proof of Theorem~\ref{thm:DieOut1} to show $(P_j)$ implies $(P_{j+1})$.
\begin{proof}
Given $k$ linearly independent null vectors and any set $J$ of $j<k$ coordinates, there is a non-zero linear combination $\nu$ of those null vectors for which $\nu_i =0$ for each $i\in J$.
\end{proof}

The proof of Theorem~\ref{thm:DieOut1} uses Lemma~\ref{lemma:die-out},
Facts~\ref{fact T is lines},~\ref{fact T}, and~\ref{fact max a min b}.

\begin{proof}[Proof of Theorem~\ref{thm:DieOut1}]
Assume $k>0$. 

By Fact~\ref{fact T is lines}, almost every $C$ satisfies $\nu \bigcdot C \ne0$ for all (non-zero) $\nu\in \cT$. Choose such a $C$.

By Fact~\ref{fact max a min b}, there exists $a^* := a(\cT)$ and $b^* := b(\cT)$ (defined in Eqs.~\eqref{a cT} and ~\eqref{b cT}) such that from Ineq.~\eqref{ineq ab lem2},
for each $\nu\in\cT$,
\begin{align*}\label{}
\min_{i: {\nu_i>0}} x_i(t)\le  e^{a_\nu - b_\nu t}\le  e^{a^* - b^*t}\text{ for all } t \ge 0.
\end{align*}

For $j\ge 1$ we define the following. 

\bigskip

\noindent
{\bf\BF Property $P_j$:} For each given $t$, there are $j$ distinct coordinate numbers $S_j :=\{i_1,\ldots,i_j\}$ such that for each $i\in S_j$, $x_i(t)\le e^{a^*-b^*t}$. 
 
We will prove $P_k$ holds by induction.
Notice that $P_1$ is satisfied by By Lemma~\ref{lemma:die-out}.
We claim that if  $1\le j<k$, then $P_j$ implies $P_{j+1}$. To see this, notice that by Fact~\ref{fact T},
there is a null vector $\nu^{(j+1)}\in \cT$ such that $\Supp(\nu^{(j+1)})$ contains none of the $j$ coordinates in $S_j$. Writing $\nu = \nu^{(j+1)}$ and applying Lemma~\ref{lemma:die-out} to $\nu$,  choose ${i_{j+1}}\in\Supp(\nu)$ such that $\nu_{i_{j+1}}>0$ and
from Ineq.~\eqref{ineq ab lem},
\begin{align*}
x_{i_{j+1}} = \min_{i: {\nu_i>0}} x_i(t)\le e^{a^*-b^*t}.
\end{align*}
Let $S_{j+1} :=\{i_1,\ldots,i_{j+1}\}$. Hence $P_{j+1}$ is satisfied, proving the claim.

By an induction that stops at $P_k$, the result is proved.
 \black
\end{proof}
\section{Examples}\label{sec examples}
We use a family of (die-out) Lyapunov functions of the form $\Lambda_\nu$ in Eq.~\eqref{lyap} to establish that $k$ species die out exponentially fast (Def.~\ref{def dieout}). 

\begin{example}[\textbf{A three-predator one-prey model where two species must die out}]\label{ex:4D}
\normalfont
Fig.~\ref{fg:4-troph c} shows a graph of an ecosystem that motivated our investigations of Lyapunov functions.
Fig.~\ref{fig:3preds1prey-dynamics} {\black(see Introduction)}  displays the behavior of the Lotka-Volterra model. 
Fig.~\ref{fig:3preds1prey-dynamics} (A) illustrates that in unequal resource distribution, the less efficient predators can die out.
Fig.~\ref{fig:3preds1prey-dynamics} (B) shows the oscillation of the population. 
 
 We can model this system in two ways, either as a 4-dimensional Lotka-Volterra system
Eq.~\eqref{eq:LV} 
 or as the three-dimensional non-autonomous system 
~\eqref {eq:Z}
where it is not specified how the $Z=(z_1)$ is determined. In the Lotka-Volterra system, we choose $x_1$ to be the prey density and $x_2,x_3,x_4$ the three predators.
Die out occurs however $Z$ is defined provided there is a bounded solution.
In order to make the coordinates in the two approaches compatible, we write the ${Z}$ system as follows so that the subscripts of the two approaches are the same, so that the matrix below $\SZ$ is a $3 \times 1$ submatrix of the  $4\times4$ Lotka-Volterra matrix $\SLV$.

\begin{align}\label{eq S}  
\left(\frac{\dot{x_2}}{x_2},\frac{\dot{x_3}}{x_3},\frac{\dot{x_4}}{x_4}\right)(t) &= C + S Z(t);\quad\text{where }
  { \SZ =\left[
\begin{array}{c}
      s_{21}  \\
      s_{31}  \\
      s_{41}
    \end{array}\right]}.
\end{align}
The number of predators that can survive depends on the nonzero rows of null vectors of the  matrix and on $C$.
The kernel coordinates are $2,3$, and $4$.
{\black There could be additional populations, $x_1, x_5, \ldots$, but as long as there is a bounded solution $X(t)$, the die-out properties of $x_2,x_3,x_4$ depend only on Eq.~\eqref{eq S}, and are independent of the form of the equations for the additional variables.}

Let $k$ be the dimension of the kernel of $S$.
For typical coefficients, the matrix $\SZ$ has a kernel of dimension $k=2$ with the following null vectors, any two of which are a basis for the kernel.
\begin{align*}
 \nu^{(34)} &= \pm\begin{bmatrix}
 0,s_{41},-s_{31}
 \end{bmatrix}, \\
 \nu^{(24)} &= \pm\begin{bmatrix}
       -s_{41}, 0, s_{21} 
     \end{bmatrix},\\
     \nu^{(23)} &= \pm\begin{bmatrix}
       -s_{31}, s_{21}, 0  
     \end{bmatrix}. 
\end{align*}

In our notation, {\black when a null vector has non-zero coordinates such as $j_1,j_2$, and $j_3$, we write it as $\nu^{(j_1j_2j_3)}$.} 
We use such notation only for null vectors with minimal support.

For almost every $C$, choose the sign of each $\nu$ above so that
$\nu\bigcdot C<0$ as required by Theorem~\ref{thm:DieOut1}.

\black
Depending on $C$, the team of die-out Lyapunov functions is:

\begin{align}
     \Lambda_{\nu^{(34)}}(X(t)) 
    &= \pm(s_{41} \ln x_3(t) -s_{31} \ln x_4(t)),\label{eq:34}\\
    \Lambda_{\nu^{(24)}}(X(t)) &= 
   \pm( -s_{41} \ln x_2(t) + s_{21} \ln x_4(t)), \\ \label{eq:32}
   \Lambda_{\nu^{(23)}}(X(t)) &= 
    \pm(-s_{31} \ln x_2(t) + s_{21} \ln x_3(t)). 
\end{align}

{\black Let $X(t)$ be a bounded trajectory.

Let $\nu^{(34)} = \begin{bmatrix}
 0,s_{41},-s_{31} 
 \end{bmatrix}$;  (we have chosen a plus sign here), and assume $s_{41}, s_{31} >0$ to simplify calculation. 
Then $\Lambda_{\nu^{(34)}}=s_{41} \ln x_3(t) -s_{31} \ln x_4(t)$; and $\Lambda_{\nu^{(34)}}(X)=s_{41} \ln x_3 -s_{31} \ln x_4$. Suppose $C \bigcdot\nu^{(34)} \neq 0$.

If $C \bigcdot\nu^{(34)} > 0$, then $\Lambda_{\nu^{(34)}}(X) \to +\infty$. Since the coordinates of $X$ are bounded, $s_{41} \ln x_3(t)$ is bounded. Hence, $-s_{31} \ln x_4(t) \to \infty$, which means $\ln x_4(t) \to -\infty$, which means $x_4(t) \to 0$, as $t \to \infty$. 

If however $C \bigcdot\nu^{(34)} < 0$, then by similar reasoning, $x_3(t) \to 0$ as $t \to \infty$. 

More generally, depending on the signs of $C \bigcdot\nu$ and $s_{ij}$, 
\begin{itemize}
    \item $\Lambda_{\nu^{(23)}}$  tells us $x_2$ or $x_3$ must die out, and 
    \item $\Lambda_{\nu^{(24)}}$  tells us $x_2$ or $x_4$ must die out, and 
    \item $\Lambda_{\nu^{(34)}}$  tells us that for any bounded solution, $x_3$ or $x_4$ must die out.
\end{itemize}
Together the three $\Lambda$'s tell us that at least two of the three populations must die out. 
}
\end{example}
{\black In this example knowing the signs of the coefficients $s_{ij}$ and the signs of the $C \bigcdot\nu$'s, then we can determine which two populations must die out.  }

\subsection{Who must die?} Sometimes a team of Lyapunov functions determines who must die {\black (as in the above example)} and sometimes it does not. We show a case below where whether it does or does not depend on the constant $C$. 
The simplest case of Eq.~\eqref {eq:Z} is where (i) the null space is one dimensional, so there is essentially one Lyapunov function, and (ii) where $Z(t)$ is constant. In the next example, the Lyapunov function is Eq.~\eqref{intro null vec}. 

\black
\begin{example}[\bf{The kernel is one dimensional and $Z$ is constant}]\label{ex specific}
\normalfont
Here we illustrate how the choice of $C$ affects which species dies out in a simple case where $Z(t)$ is constant. Consider the following three-dimensional model. 
\begin{align*}
\begin{split}
\left[
\begin{array}{c}
      \frac{\dot{x}_1}{x_1} \\
      \frac{\dot{x}_2}{x_2}  \\
      \frac{\dot{x}_3}{x_3}
    \end{array}\right] = \left[
\begin{array}{c}
      -1+z_1+2z_2 \\
      c_2+z_1+z_2 \\
      -1+3z_1+z_2
    \end{array}\right] = \left[
\begin{array}{c}
      -1 \\
      c_2  \\
      -1
    \end{array}\right] + \left[
\begin{array}{cc}
      1 & 2 \\
      1 & 1  \\
      3 & 1
    \end{array}\right] \left[
\begin{array}{c}
      z_1 \\
      z_2 
    \end{array}\right]
\end{split}
\end{align*}
in which $Z=(z_1,z_2)$ is constant and $C=(-1,c_2,-1)$ has one free parameter, $c_2$, which can be thought of as the per capita death rate of $x_2$.

For an ecological interpretation, we might call $x_1, x_2,$ and $x_3$ predators and $z_1$ and $z_2$ are prey or resources. 

We explore how the behavior depends on $c_2$. Assume there is a bounded solution $X(t)$ in $\cP$ (with all coordinates strictly positive).
Define three half-planes,
\begin{align} \label{ineq:3d}
\begin{split}
\cH_1 &:= \{(z_1,z_2):-1+z_1+2z_2\le 0\} \text{ where }  \dt x_1 \le 0, \\
\cH_2 &:= \{(z_1,z_2):c_2+z_1+z_2\le 0\} \text{ where }  \dt x_2 \le 0,\\
\cH_3 &:= \{(z_1,z_2):-1+3z_1+z_2\le 0\} \text{ where }  \dt x_3 \le 0.
\end{split}
\end{align}
If $Z$ is chosen in the interior of {\black half-plane} $\cH_j$, then $\frac{\dt x_j}{x_j}$ is negative and is constant since $Z$ is constant, so $x_j(t) \to 0$ exponentially fast.

{\black For simplicity we assume $z_1,z_2 \in \cP$, and we plot $G \cap \cP$ }be the shaded set in Fig.~\ref{fig:3 half planes} where the three planes and $\cP$ intersect. 
For there to be a bounded trajectory, $Z$ must lie in $G$ and quite possibly on the boundary of $G$.
The boundary line of each of the above 3 half planes is shown, and the shaded region shows the closed set $G$ in $\cP$ where all three inequalities are satisfied. The dashed line is the boundary of half-plane $(\cH_2)$.

Since in this example we assume $Z$ is constant, if $Z=(z_1,z_2)$ is outside $G$, all species must die out. When $Z\in\cP$ is a vertex of $G$, two species can persist. 
Write 
$$S =\left[
\begin{array}{cc}
      1 & 2  \\
      1 & 1  \\
      3 &1
    \end{array}\right].$$
The kernel of $S^T$ is one dimensional, and it contains the null vector $(2,-5,1)^T$. Its coordinates are the coefficients of Eq.~\eqref{intro null vec}, and $C = (-1,c_2,-1)$.

 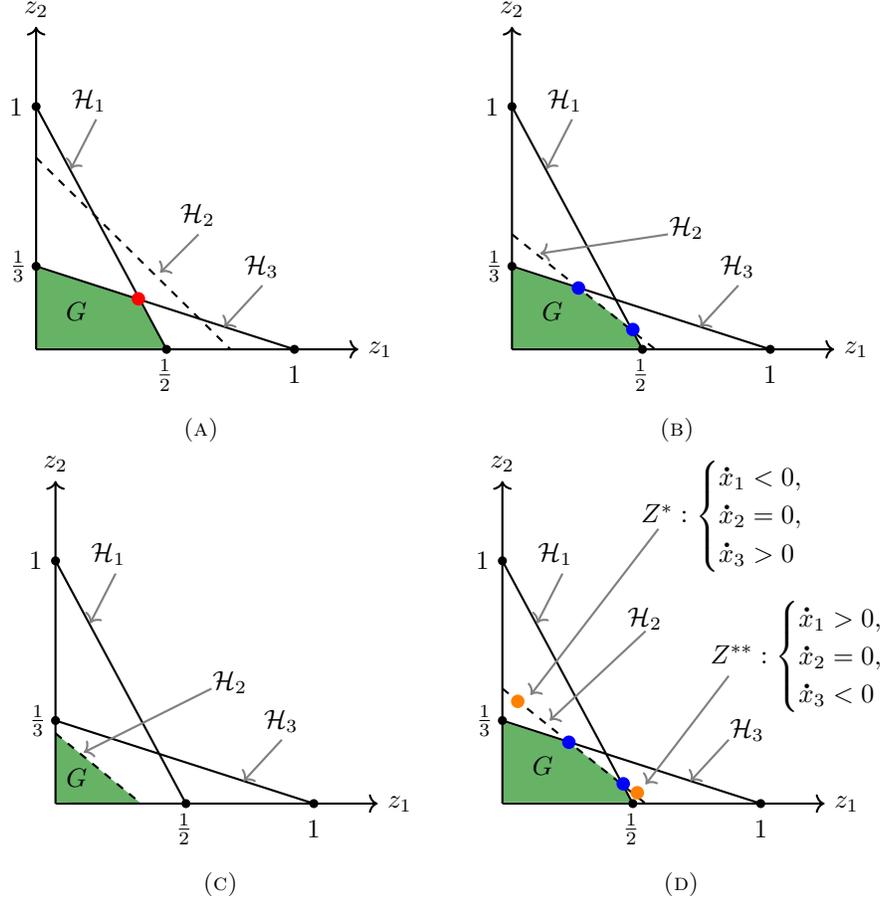
\begin{figure}[H]
\begin{minipage}[b]{0.5\linewidth}  
\centering
\begin{tikzpicture}[scale=.85,sibling distance=0pt]
     \pgfmathsetmacro{\e}{1.4}   
    \pgfmathsetmacro{\a}{1}
     \fill[dgn!60] (-3.54,-3.78) -- (-3.54,-2.52) -- (-1.95,-3.009) -- (-1.5,-3.78) -- cycle; 
    \draw[thick, ->] (-3.54,-3.8) -- (1.5,-3.8) node[right]{$z_1$};    
    \draw[thick, ->] (-3.54,-3.8) -- (-3.54,1.25) node[above]{$z_2$};   
    \draw[thick, dashed] (-3.54,-0.8) -- (-.5,-3.8); 
     
     \draw (-3.54,-2.5) node[left]{$\frac{1}{3}$};
     \fill (-3.54,-2.5)  circle[radius=2pt];
     \draw[thick] (-3.54,-2.5) -- (0.5,-3.8);
     \draw (0.75,-4.2) node[left]{$1$};
     \fill (0.5,-3.8)  circle[radius=2pt];
     
     \draw (-1.25,-4.2) node[left]{$\frac{1}{2}$};
     \fill (-1.5,-3.8)  circle[radius=2pt];
     
     \draw (-3.6,.0) node[left]{$1$};
     \fill (-3.54,.0)  circle[radius=2pt];
     \draw[thick] (-3.54,.0) -- (-1.5,-3.8);
     
     \fill [red] (-1.94,-3.009)  circle[radius=3pt];

     \draw (-2.6,-3.2) node[left]{$G$};
     \draw (-2.3,0.1) node[left]{$\cH_1$};
     \path [->,draw,thick,gray] (-2.6, -.2) -- (-3.,-0.98);
     \draw (0.4,-2.5) node[left]{$\cH_3$};
     \path [->,draw,thick,gray] (0, -2.8) -- (-0.6,-3.45);
     \draw (-.6,-1.7) node[left]{$\cH_2$};
     \path [->,draw,thick,gray] (-1, -2.) -- (-1.6,-2.6);
\end{tikzpicture}
 \subcaption{}
\end{minipage}
\begin{minipage}[b]{0.5\linewidth}  
\centering
\begin{tikzpicture}[scale=.85,sibling distance=0pt]
     \pgfmathsetmacro{\e}{1.4}   
    \pgfmathsetmacro{\a}{1}
     \fill[dgn!60] (-3.54,-3.78) --(-3.54,-2.52) -- (-2.5,-2.85) -- (-1.75,-3.41) -- (-1.5,-3.78) -- cycle; 
    \draw[thick, ->] (-3.54,-3.8) -- (1.5,-3.8) node[right]{$z_1$};    
    \draw[thick, ->] (-3.54,-3.8) -- (-3.54,1.25) node[above]{$z_2$};   
    \draw[thick, dashed] (-3.54,-2) -- (-1.3,-3.8); 
 
     \draw (-3.54,-2.5) node[left]{$\frac{1}{3}$};
     \fill (-3.54,-2.5)  circle[radius=2pt];
     \draw[thick] (-3.54,-2.5) -- (0.5,-3.8);
     \draw (0.75,-4.2) node[left]{$1$};
     \fill (0.5,-3.8)  circle[radius=2pt];
     
     \draw (-1.25,-4.2) node[left]{$\frac{1}{2}$};
     \fill (-1.5,-3.8)  circle[radius=2pt];
     
     \draw (-3.6,.0) node[left]{$1$};
     \fill (-3.54,.0)  circle[radius=2pt];
     \draw[thick] (-3.54,.0) -- (-1.5,-3.8);
     
     \fill [blue] (-2.5,-2.84)  circle[radius=3pt];
     \fill [blue](-1.65,-3.49)  circle[radius=3pt];
     
     \draw (-2.6,-3.2) node[left]{$G$};
     \draw (-2.3,0.1) node[left]{$\cH_1$};
     \path [->,draw,thick,gray] (-2.6, -.2) -- (-3.,-0.98);
     \draw (0.4,-2.5) node[left]{$\cH_3$};
     \path [->,draw,thick,gray] (0, -2.8) -- (-0.6,-3.45);
     \draw (-.4,-1.9) node[left]{$\cH_2$};
     \path [->,draw,thick,gray] (-1.1, -2.) -- (-3.1,-2.3);
\end{tikzpicture}
 \subcaption{}
\end{minipage}
\begin{minipage}[b]{0.5\linewidth} 
\centering
\begin{tikzpicture}[scale=.85,sibling distance=0pt]
     \pgfmathsetmacro{\e}{1.4}   
    \pgfmathsetmacro{\a}{1}
     \fill[dgn!60] (-3.54,-3.78) --(-3.54,-2.7) -- (-2.2,-3.8) -- cycle; 
    \draw[thick, ->] (-3.54,-3.8) -- (1.5,-3.8) node[right]{$z_1$};    
    \draw[thick, ->] (-3.54,-3.8) -- (-3.54,1.25) node[above]{$z_2$};   
    \draw[thick, dashed] (-3.54,-2.7) -- (-2.2,-3.8); 
 
     \draw (-3.54,-2.5) node[left]{$\frac{1}{3}$};
     \fill (-3.54,-2.5)  circle[radius=2pt];
     \draw[thick] (-3.54,-2.5) -- (0.5,-3.8);
     \draw (0.75,-4.2) node[left]{$1$};
     \fill (0.5,-3.8)  circle[radius=2pt];
     
     \draw (-1.25,-4.2) node[left]{$\frac{1}{2}$};
     \fill (-1.5,-3.8)  circle[radius=2pt];
     
     \draw (-3.6,.0) node[left]{$1$};
     \fill (-3.54,.0)  circle[radius=2pt];
     \draw[thick] (-3.54,.0) -- (-1.5,-3.8);
     
     \draw (-2.9,-3.4) node[left]{$G$};
     \draw (-2.3,0.1) node[left]{$\cH_1$};
     \path [->,draw,thick,gray] (-2.6, -.2) -- (-3.,-0.98);
     \draw (0.4,-2.5) node[left]{$\cH_3$};
     \path [->,draw,thick,gray] (0, -2.8) -- (-0.6,-3.45);
     \draw (-.4,-1.9) node[left]{$\cH_2$};
     \path [->,draw,thick,gray] (-1.1, -2.) -- (-3.1,-3.);
\end{tikzpicture}
 \subcaption{}
\end{minipage}
\begin{minipage}[b]{0.45\linewidth}  
\centering
\begin{tikzpicture}[scale=.85,sibling distance=0pt]
     \pgfmathsetmacro{\e}{1.4}   
    \pgfmathsetmacro{\a}{1}
     \fill[dgn!60] (-3.54,-3.78) --(-3.54,-2.52) -- (-2.5,-2.85) -- (-1.75,-3.41) -- (-1.5,-3.78) -- cycle; 
    \draw[thick, ->] (-3.54,-3.8) -- (1.5,-3.8) node[right]{$z_1$};    
    \draw[thick, ->] (-3.54,-3.8) -- (-3.54,1.25) node[above]{$z_2$};   
    \draw[thick, dashed] (-3.54,-2) -- (-1.3,-3.8); 
 
     \draw (-3.54,-2.5) node[left]{$\frac{1}{3}$};
     \fill (-3.54,-2.5)  circle[radius=2pt];
     \draw[thick] (-3.54,-2.5) -- (0.5,-3.8);
     \draw (0.75,-4.2) node[left]{$1$};
     \fill (0.5,-3.8)  circle[radius=2pt];
     
     \draw (-1.25,-4.2) node[left]{$\frac{1}{2}$};
     \fill (-1.5,-3.8)  circle[radius=2pt];
     
     \draw (-3.6,.0) node[left]{$1$};
     \fill (-3.54,.0)  circle[radius=2pt];
     \draw[thick] (-3.54,.0) -- (-1.5,-3.8);
     
     \fill [blue] (-2.5,-2.84)  circle[radius=3pt];
     \fill [blue](-1.65,-3.49)  circle[radius=3pt];

     \fill [orange] (-3.3,-2.2)  circle[radius=3pt];
     \fill [orange](-1.43,-3.63)  circle[radius=3pt];
     
     \draw (-2.6,-3.2) node[left]{$G$};
     \draw (-2.3,0.1) node[left]{$\cH_1$};
     \path [->,draw,thick,gray] (-2.6, -.2) -- (-3.,-0.98);
     \draw (0.7,-2.65) node[left]{$\cH_3$};
     \path [->,draw,thick,gray] (0, -2.8) -- (-0.6,-3.45);
     \draw (-.9,-.9) node[left]{$\cH_2$};
     \path [->,draw,thick,gray] (-1.3, -1.1) -- (-2.8,-2.5);
     \draw (1.75,0.7) node[left]{$Z^*: \begin{cases} \dt x_1 <0, \\ \dt x_2 =0, \\ \dt x_3 >0 \end{cases}$};
     \path [->,draw,thick,gray] (-1.1, .5) -- (-3.1,-2.1);
     \path [->,draw,thick,gray] (0, -1.8) -- (-1.3,-3.5);
     \draw (3.,-1.5) node[left]{$Z^{**}: \begin{cases} \dt x_1 >0, \\ \dt x_2 =0,\\ \dt x_3 <0 \end{cases}$};
\end{tikzpicture}
 \subcaption{}
\end{minipage}
 \caption{
 {\bf Where all species die out, Example~\ref{ex specific}.} The shaded (green) area is the intersection of $\cP$ and the three half-planes $\cH_1,\cH_2,$ and $\cH_3$ where the inequalities \eqref{ineq:3d} are all satisfied. {\black For simplicity, we assume $z_1,z_2 \in \cP$ and we only examine the part of $G$ that is in $\cP$.}
 The arrow for each $i$  points to the boundary of  $\cH_i$, where 
 $\dot x_i = 0$.
 In the interior of $G$, the inequalities are all negative, and all species die out. On the boundary of $G$, one species can survive. At each vertex of $G \in \cP$, two species can survive.
{\bf{Panel (A):}} In the green region in panel (A), $c_2 < - \frac{3}{5}$ and the inequality in the definition of $(\cH_2)$ is always negative so $x_2\to 0$ as $t\to\infty$,
 and there is one vertex (red dot) where $Z$ is on the boundary of two half-planes. There, $x_1$ and $x_3$ can both persist for that value of $(z_1,z_2)$. 
{\bf{Panel (B):}} There are two vertices (blue dots) on the  interior boundary of $G$. At one (upper blue dot), $x_1$ and $x_2$ can persist while at the other (lower blue dot) $x_2$ and $x_3$ can persist. {\bf{Panel (C):}} Only species $x_2$ survives.  {\bf{Panel (D):}} $Z^*$ and $Z^{**}$ are points on the dashed line where $\dot x_2=0$.}
 \label{fig:3 half planes}
\end{figure}


The boundaries of $\cH_1$ and $\cH_3$ intersect at $(\frac{1}{5},\frac{2}{5})$, which is in $\cH_2$ if $c_2 \le -\frac{3}{5}$. Notice $\nu\bigcdot C = 0$ is satisfied when the boundaries of the three half-planes intersect at one point (see Fig.~\ref{fig:3 half planes}). 

 {\bf When the die-out conclusion of Theorem \ref{thm:DieOut1} can fail and coexistences occurs:} $(2,-5,1)\cdot C=0$, \ie, $c_3 = -\frac{3}{5}$. For that $C$, the borders of $H_1, H_2$, and $H_3$ intersect, \ie, all $\dt x_i$ are $0$ at the one point where $z_1 = \frac{1}{5}, z_2 = \frac{2}{5}$. There the three species are constant. Any positive values of $x_1, x_2,x_3$ suffice. 

{\bf\BF Two cases, either $(2,-5,1)\bigcdot C = -3-5c_2 <0$ or $>0$.} In Fig.~\ref{fig:3 half planes}, panel (A) is for the case where $(2,-5,1)\bigcdot C <0$ and panel (B) is for $(2,-5,1)\bigcdot C > 0$. 

For panel (A), there is one vertex on the boundary of $G$ in $\cP$, and that corresponds to the vertex where both $x_1$ and $x_3$ survive while $x_2$ dies out.

\bf\BF {When neither $x_1$ nor $x_2$ dies out but $\min\{x_1(t), x_2(t)\} \to 0$ as $t \to \infty$.}
For panel (B) of Figure~\ref{fig:3 half planes} there are two vertices, the only two possible constant values of $Z$ for which two species persist. 
Both are on the boundary of $\cH_2$. In either case, $x_2$ persists while either $x_1$ or $x_3$ dies out.

In panel (C), both $x_1$ and $x_3$ die out. 
\end{example}

In Figure~\ref{fig:3 half planes}, points $Z^*$ and $Z^{**}$ both have $\dt x_2 = 0$; $Z^*$ has $\dt x_1<0$ and $\dt x_3>0$; and $Z^{**}$ has the reverse. 

By having resource $z(t)$ oscillate between levels $Z^*$ and $Z^{**}$, it is possible to have $x_1(t)$ and $x_2(t)$ oscillate with neither going to $0$, but $\min\{x_1(t),x_2(t)\}\to 0$ exponentially fast. The simplest case is where $z(t)$ discontinuously jumps back and forth between $Z^*$ and $Z^{**}$. 
Suppose $x_1(0)$ and $x_3(0)$ are less than some $m>0$. Suppose $z(t) = Z^{**}$ for an interval $J_1 = [0,T_1]$ where $T_1$ is chosen so that $x_1(t)$ rises to $m$ at $T_1$. Of course, $x_3$ is decreasing on $J_1$. 

Then on $J_2 = (T_1,T_2]$, $z(t) = Z^*$, and $T_2$ is chosen so that $x_3(t)$ rises to $m$ at $t=T_2$. 
We can keep oscillating between $Z^*$ and $Z^{**}$ so that one the two, $x_1$ or $x_3$ returns to $m$ at each interval end, $T_j$. This cannot be periodic became each $x_i$ must rise $T_{j+1}-T_j \to \infty$. 

\begin{example}[{\bf Five predators and two prey}]\label{ex: 2 7D}
 \normalfont
We could consider the $7$-dimensional LV model for five predators and two prey with the matrix.
Assume $\SLV = (s_{ij})$ is the corresponding matrix.
\begin{align}\label{S 7x7}
S_{LV}=\left[
    \begin{array}{ccccccc}
        s_{11} & s_{12} & s_{13} & s_{14}& s_{15} & s_{16} & s_{17} \\
        s_{21} & s_{22} & s_{23} & s_{24}& s_{25} & s_{26} & s_{27} \\ 
   s_{31} & s_{32} & 0 & 0 & 0 & 0 & 0\\
   \vdots & \vdots &  \vdots & \vdots &  \vdots & \vdots &  \vdots \\
   s_{71} & s_{72} & 0 & 0 & 0 & 0 & 0
    \end{array}
\right].
\end{align}
{\black The dimension of the kernel of this matrix is $3$. For almost every choice of $s_{ij}$, every null vector ($\nu_i$) of $S_{LV}^T$ can only have $\nu_i \neq 0$ if $i \in \{3,4,5,6,7\}$.} To show which species must die out (for a bounded trajectory), we need only consider Eq.~\eqref{eq:Z} whose matrix is
\begin{align*}
S_Z=\left[
    \begin{array}{cc}
   s_{31} & s_{32} \\
  \vdots & \vdots \\
   s_{71} & s_{72} 
    \end{array}
\right].
\end{align*}
{\black This is the smallest submatrix of $S_{LV}^T$ where kernel has dimension $3$.}
{\black Because each null vector of $S_{LV}^T$ can have non-zero coordinates $3, \ldots, 7$ and $S_{LV}$ maps those coordinates to coordinates $1$ and $2$. In other words, $S_Z$ captures the null space behavior of $S_{LV}$ and ignores the rest of the $S_{LV}$.}

This transpose matrix $S_Z^T$ has a 3-dimensional kernel for almost every choice of non-zero entries.  The coordinates of $Z = (z_1, z_2)$ correspond to either prey species or ``resources''. {\black These values can even be negative.}
Theorem~\ref{thm:DieOut1} asserts that at least three $x_i$ must die out exponentially fast for almost every $C$. Since the kernel coordinates are $3,4,5,6,7$, three of these must die out. These correspond to the predators in the biological interpretation.
Four other species, numbers 1 and 2, and two kernel coordinates can coexist for some choices of $S$, $C$, and $Z(t)$.

{\bf\BF Null vectors for Eq.~\eqref{S 7x7} and the Die-out Lyapunov functions. }
For almost every choice of the non-zero coefficients of $S$, the kernel is three-dimensional, and for those cases, the $\SLV^T$ null vectors have the form $(0,0,*,*,*,*,*)$ for model $LV$ where each `$*$' indicates a coefficient that can be non-zero. 
For Eq.~\eqref{eq:Z}, the first two coordinates are omitted, so
the $\SZ^T$ null vectors have the reduced form $(*,*,*,*,*)$ in which the coordinate numbers are still $\{3,4,5,6,7\}$.

Assuming the kernel is typically $k$-dimensional, we can choose any $k-1$ coordin\-ates, and, taking a linear combination of the kernel vectors, we can create a non-zero null vector whose entries for those $k-1$ coordinates are $0$, (Fact \ref{fact T}). It is possible that some other entries would also be 0. 
Since here the kernel is three-dimensional, for any two kernel coordinates $i$ and $j$, there is a non-zero null vector $\nu$ for which $\nu_i=0$
and $\nu_j=0$, leaving at most  three non-zero coordinates.


Let the $i^{th}$, $j^{th}$ and $u^{th}$ coordinates be non-zero for $i,j,u \in\{3,4,5,6,7\}$.  
There are $\binom{5}{3}=10$ such choices of null vectors that have $3$ non-$0$ coordinates. We display three of these as samples:
\begin{align*}
\begin{split}
    \nu^{(367)} &= \pm[p_3,0,0,p_6,p_7],  \\
    \nu^{(467)} &= \pm[0,q_4,0,q_6,q_7],   \\
    \nu^{(567)} &= \pm[0,0,r_5,r_6,r_7],  
    \end{split}
\end{align*}
where
\begin{equation*}
  \begin{split}
    p_3 &= s_{61}s_{72} - s_{62}s_{71}, \\
    p_6 &= s_{32} s_{71} - s_{31}s_{72}, \\
    p_7 &= s_{31}s_{62} - s_{32} s_{61},
  \end{split}
\ \
\&
\ \ 
  \begin{split}
    q_4 &= s_{61}s_{72} - s_{62}s_{71}, \\
    q_6 &= s_{42} s_{71} - s_{41}s_{72}, \\
    q_7 &= s_{62}s_{41} - s_{61}s_{42}. 
  \end{split} 
\ \
\&
\ \ 
    \begin{split}
    r_5 &= s_{61}s_{72} - s_{62}s_{71}, \\
    r_6 &= s_{52}s_{71} - s_{51}s_{72}, \\
    r_7 &= s_{62}s_{51} - s_{61}s_{52}. 
  \end{split}
\end{equation*}
The sign $\pm$ of $\nu$ above is chosen so that $\nu \bigcdot C < 0$, a choice which is possible for almost every $C$, (see Fact \ref{fact T is lines}). Below we omit the $\pm$, leaving it to the reader.


These three yield the following are the die-out Lyapunov functions:   
    \begin{align*}
    \begin{split}
    \Lambda_{\nu^{(367)}}(x_3,x_6,x_7) &= p_3 \ln x_3 + p_6 \ln x_6 + p_7 \ln x_7  \\
    \Lambda_{\nu^{(467)}}(x_4,x_6,x_7) &= q_4 \ln x_4 + q_6 \ln x_6 + q_7 \ln x_7,  \\
    \Lambda_{\nu^{(567)}}(x_5,x_6,x_7) &= r_5 \ln x_5 + r_6 \ln x_6 + r_7 \ln x_7. \end{split}
    \end{align*}
for which we obtain
\begin{align*}
\begin{split}
    \dt \Lambda_{\nu^{(367)}} &= p_3c_3 + p_6 c_6 + p_7 c_7, \\ 
     \dt \Lambda_{\nu^{(467)}} &= q_4c_4 + q_6 c_6 + q_7 c_7, \\
     \dt \Lambda_{\nu^{(567)}} &= r_5c_5 + r_6 c_6 + r_7 c_7,
     \end{split}
    \end{align*}
{\black Each of the ten  $\Lambda_{\nu}$ with only three non-zero coordinates tells us that at least one of its three variables must die out. Knowledge of the signs of the coefficients and of $\Lambda \bigcdot C$ will give additional information as to which must die out. Together they guarantee that at least three species must die out for almost every $C$. }   

We note that we can also write the die-out Lyapunov functions in the form 
\begin{align*}
    V(X)=:\exp(\Lambda_{\nu^{(367)}}(X)) =:x_3^{p_3} x_6^{p_6} x_7^{p_7}. 
\end{align*}
though then $\dt V(x)$ is not constant.

\end{example}

\begin{example}[{\bf {A $14$-species ecosystem and its graph,
Fig.~\ref{fg:14-trophic}}}]\label{ex:26D}
\label{sec:example 3}
\normalfont 
This example is closely related to Example~\ref{ex: 2 7D} in that the $S$ in both examples have a 3-dimensional kernel, so at least 3 species must die out.

In Fig.~\ref{fg:14-trophic},  there are 14 nodes or species, but when we use Eq.~\eqref{eq:Z}, we only use $d=8$ equations.  Let $X=(x_1,\dots,x_8)$ represent the densities of the 8 blue nodes $\{4,5,\ldots, 11\}$;
and $Z=(z_j)_{j=1}^5$ (so $d' =5$) corresponding to the red nodes  $\allowbreak\{1,2,3,12,13\}$.
Each of the edges from 8 blue nodes end at one of the 5 red nodes.
The $8$ blue species depend only on the $5$ red species, and it is this aspect that we capture in Eq.~\eqref{eq:Z}. 
Applying  Eq.~\eqref{eq:Z} to model the graph, the matrix $S$ is $8 \time 5$ and so has a kernel whose dimension is at least $3$. 
We conclude that 
of the 8 kernel nodes shown in blue,  at least $3$ must die out.
Note that some red nodes are in a trophic level higher than blue nodes while others are below.

For almost every choice of the matrix $S_Z$ there are $\binom{8}{6}=28$ ways of choosing a minimal-support null vector $\nu$ of $\SZ^T$  so that it has only $6$ non-zero coordinates. Hence, there are $28$  die-out Lyapunov functions, $\Lambda_{\nu}$.
For almost every $C$, at least three of the corresponding $8$ species must die out. This is generally not enough information to indicate which three of the $8$ must die out, but it gives a lot of hints. Which species die out depends on the specific values of $C, S$, and $Z(t)$.  
\end{example}

\section{The existence of bounded solutions and a trapping region for (generalized) Lotka-Volterra models}\label{sec:trapping region}

Theorem~\ref{thm:trap} {\black gives conditions that guarantee all solutions are bounded}.
In this Section, we use a different kind of Lyapunov function that we call a trapping-region Lyapunov function. Then we use it to prove Theorem~\ref{thm:trap}. 

In contrast with die-out Lyapunov functions, our second application of Lyapunov functions is to establish that for what we call ``trophic'' Lotka-Volterra systems, there is a globally attracting trapping region. \cite{lorenz1963deterministic} created such a ``trapping Lyapunov function'' to show his famous differential equations system has a globally attracting trapping region.
It says that if a Lotka-Volterra system satisfies our ``trophic'' condition, Def.~\ref{trophic-def}, then there is a bounded globally attracting trapping region, a region that solutions cannot leave, 
and all solutions are bounded.

We create a function $V$ of the form  $V(X) = \sum_{j=1}^d \varepsilon^j x_j$ for some $\varepsilon>0$, such that (i) there is  a positive constant $\lambda$ for which
$\dt V(X(t)) < 0$ whenever $V(X(t)) \ge \lambda$, and (ii) $V(X)\to\infty$ as $|X|\to\infty$. Then the set of $X$ for which $V(X) \le \lambda$ is a bounded trapping region.

The proof  of Theorem~\ref{thm:trap} follows from the following proposition and a lemma. 
 
\begin{figure}[htbp]
\begin{center}
\centering
\begin{tikzpicture}
\shadedraw [bottom color=gray!60, top color=gray!60]
               (-3.54,-0.7) --++ (0.,-3.1) --++ (2.06,0) -- cycle;
    \pgfmathsetmacro{\e}{1.4}   
    \pgfmathsetmacro{\a}{1}
    \draw[thick, ->] (-4.5,-3.8) -- (5.5,-3.8) node[right]{$x_i$};    
    \draw[thick, <-,  line width=0.75mm] (2,-3.8) -- (3,-3.8) node[right]{};    
    \draw[thick, ->] (-3.54,-4.5) -- (-3.54,2) node[above]{$x_j$};   
    \draw[thick, ->,  line width=0.75mm] (-3.54,0.5) -- (-3.54,-0.5) node[above]{};   
    \draw[thick, ->,  line width=0.75mm] (0.4,-0.5) -- (-0.4,-1.5) node[above]{};   
    \draw[thick] (-3.54,0.7) -- (4.5,-3.78);   
    \draw[thick] (-3.54,-0.7) -- (-1.5,-3.78); 
    \draw[thick] (-1.5,1) -- (-3.54,-3.78); 
    \draw (1,0) node[right]{$\dt V<0$};
    \draw (1.8,-2) node[right]{$V=\lambda$};
    \draw (-1.,-2.8) node[right]{$\Gamma_\lambda$};
    \draw (-3.2,-3.2) node[right]{$\dt V>0$};
    \draw (-2.9,1) node[right]{$\mathcal E= \nabla V$};
\end{tikzpicture}
 \caption{ Assume $\lambda$ sufficiently large that when $\protect \dt V(X) \le 0$, $V(X)< \lambda$. 
Then $\Gamma_{\lambda} = \{X: V(X) \le \lambda\}$ is a globally attracting trapping region and no trajectories in it can leave it. Every trajectory in $\cP$ eventually enters it. If there is a fixed point $X_0$ in $\cP$, then $X_0$ must be on the boundary of the shaded region.}
 \label{fig:4}
\end{center}
\end{figure}
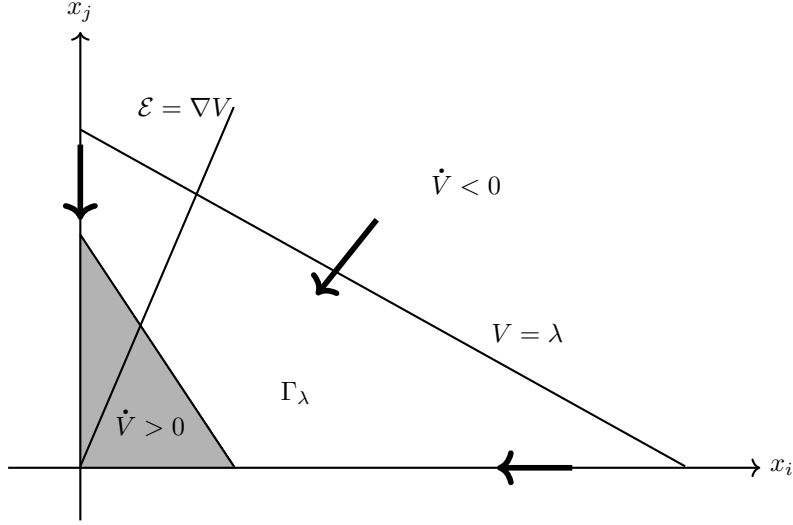

\begin{prop}[{\bf{A trapping Lyapunov function $V$ for trophic systems}}] \label{TLF}
Assume system Eq.~ \eqref{eq:LV} is trophic.
Consider system \eqref{eq:LV}.
Let $\mathcal E := (\varepsilon, \varepsilon^2, \ldots, \varepsilon^{d})$.
Let  
\begin{align}\label{vdot}
    V(X) := \sum_{n=1}^d \varepsilon^{n}x_n = \mathcal E \bigcdot X.
\end{align}
Then for some $\varepsilon>0$ and $\lambda >0$, $\dt V(X) <0$ when $V(X) \ge \lambda$  (See Fig.~\ref{fig:4}). Define
\begin{align*}
    \Gamma_\lambda := \{X \in \bcP : V(X) \le \lambda\}.
\end{align*}
Then $\Gamma_\lambda$ is a bounded globally attracting trapping region for $\bcP$.
\end{prop}

{\bf A Puzzle.} The choice of the vector $\mathcal E$ in Eq.~\eqref{vdot} is far from optimal. 
The trapping regions are not as small as possible. 
The reader may wish to find better choices that better approximate the best trapping region.

The proof will show there exist constants $\alpha>0$ and $\beta >0$ such that
\begin{align*}
\dt V(X) \le \alpha - \beta V(X).
\end{align*}
For any initial point $x_0 \in \bar\cP$ at $t=0$, we can find an upper bound $v(t)$ for $V(X(t))$ for $t \ge0$ by setting $v(0) = V(X(0))$ and $v' = \alpha-\beta v$. 
Hence $v(t) \to \frac{\alpha}{\beta}$ as $t \to \infty$, so $\limsup_{t\to \infty} V(X(t)) \le \frac{\alpha}{\beta}$. That identifies the trapping region $V \le \lambda$ where $\lambda > \frac{\alpha}{\beta}$.

\begin{lemma}\label{lem:quadratic}
Assume $c \ge 0$ and $s < 0$. Then there exists an $a$, namely, $a =-\frac{(c+ 1)^2}{4s} >0$, such that for all $x$, $c x + s x^2 \le a - x$.
\end{lemma}

\begin{proof}
Assume $c \ge 0$ and $s < 0$. We need to show
\begin{align*}
    cx+sx^2 +x + \frac{(c+ 1)^2}{4s} \le 0,
\end{align*}
which is true since the left-hand side is
\begin{align*}
    \frac{1}{s}\Big(sx + \frac{c+1}{2}\Big)^2 \le 0.
\end{align*}
\end{proof}

\begin{proof}[Proof of Prop.~\ref{TLF}]
If $s_{nm}>0$, then by {\it{(T2)}}, $n>m$. Hence we can choose $\varepsilon \in (0,1)$ so that for all $n$ for which $s_{nm}>0$,
\begin{align*}
  \varepsilon^{(n-m)} s_{nm} + s_{mn} \le 0.
\end{align*}
For $x_n'$ in Eq.~\eqref{eq:LV} and $V$ in \eqref{vdot} $\dt V = \sum_{n=1}^d \varepsilon^{n} x_n'$. Let $\mathcal M := \{n: c_n \ge 0\}$. Then for $X \in \cP$, we can write $\dt V(X)= W_1(X) + W_2(X) + W_3(X)$ where

\begin{align*}
      W_1(X) &= \sum_{n \in \mathcal M} \varepsilon^{n} (c_n x_n  +
      s_{nn} x_n^2), \\
      W_2(X) &= \sum_{n \notin \mathcal M} \varepsilon^{n} (c_n x_n  +
      s_{nn} x_n^2) \le \sum_{n \notin \mathcal M} \varepsilon^{n} c_n x_n \le 0
\end{align*}
{\black from {\it{(T1)}} since $c_n<0$ for $n \notin \mathcal M$. }      
\begin{align*}
      W_3(X) &= \sum_{n=1}^d \sum_{\substack{m=1 \\ m\neq n}}^{d} \varepsilon^{n} s_{nm}x_n x_m \le  \sum_{m=1}^{d-1} \sum_{{n>m}}^{d} [\varepsilon^{n-m} s_{nm} + s_{mn}] x_n x_m \le 0
\end{align*}
{\black from {\it{(T2)}} since $n>m$ implies $s_{mn}\le0$, and if $s_{nm}>0$, then $s_{mn}<0$, so for $\varepsilon$ sufficiently small, $\varepsilon^{n-m}s_{nm}+s_{mn}<0$. }

Hence $V(X) \le W_1(X)$. In $W_1(X)$, since $n \in \mathcal M$, $c_n \ge 0$, so $s_{nn}<0$. By Lemma \ref{lem:quadratic}, there exists $a_n >0$ such that 
\begin{align*}
    c_n x_n  +
      s_{nn} x_n^2 \le a_n - x_n,
\end{align*}
Then
\begin{align*}
    W_1(X) \le \sum_{n \in \mathcal M} \varepsilon^{n} (a_n - x_n).
\end{align*}

Let
\begin{align*}
e_n =
    \begin{cases}
     \varepsilon^{n} \ \  &\mbox{ if } n \in \mathcal M, \\
    - \varepsilon^{n} c_n. \ \ &\mbox{ if } n \notin \mathcal M.
    \end{cases}
\end{align*}
{\black Notice $e_n>0$ for all $n$ when $\varepsilon>0$ since $c_n<0$ for $n\notin \mathcal M$.} Let $A :=\sum_{n \in \mathcal M} \varepsilon^{n}a_n$. {\black If $\mathcal M$ is the empty set, let $A :=0$.}
\begin{align*}
\begin{split}
    \dt V(X) &\le \sum_{n \in \mathcal M} \varepsilon^{n} (a_n - x_n) + \sum_{n \notin \mathcal M} \varepsilon^{n} c_n x_n \\
    & \le \sum_{n \in \mathcal M} \varepsilon^{n} (a_n - x_n) -  \sum_{n \notin \mathcal M} e_n x_n  \\
    &\le \sum_{n \in \mathcal M} \varepsilon^{n}a_n -  \sum_{n \in \mathcal M} \varepsilon^{n} x_n - \sum_{n \notin \mathcal M} e_n x_n
    \le A - \sum_{n=1}^d e_n x_n. 
    \end{split}
\end{align*}

{\black For some sufficiently small $B>0$}, $\sum_{n=1}^d e_n x_n \ge BV(X)$, so $\dt V(X) \le A - BV(X)$.
Choose $\lambda_0$ so that $A-B\lambda_0 =0$.
Let $\lambda > \lambda_0$. Then $\dt V(x) <0$ when $V(x) \ge \lambda$, so  
\begin{align*}
    \Gamma_\lambda := \{X \in \bcP : \mathcal E \bigcdot X \le \lambda\}
\end{align*}
is a bounded global trapping region for $\bcP$.
\end{proof}

{\black If $\mathcal M$ is the empty set, \ie, $c_n <0$ for all $n$, then $A=0$, and $\dt V(X)<0$ for all $X\neq0$ in 
$\bar \cP$. Hence, for every solution $X$, $X(t) \to 0$ as $t \to \infty$. }

\begin{proof}[Proof of Theorem~\ref{thm:trap}]
By Prop.~\ref{TLF}, there exists a $\bcP$ globally attracting trapping region $\Gamma_{\lambda}$, for some $\lambda$. No trajectory can leave the region $\Gamma_{\lambda}$.
Hence, each trajectory is bounded. By the LaSalle-Barbashin-Krasovskii method (see \cite{alligood1996chaos}), if $\dt V(X(t))<0$ for all $t$, then $\dt V = 0$ at each limit point of $X$, 
so $X(t) \to \Gamma$ since all limit points of $X(t)$ are in $\Gamma$, and if there exists $t_0 \ge 0$ for which $\dt V(X(t_0))\ge 0$, then $X(t_0) \in \Gamma_{\lambda}$, and $X(t)$ remains there for $t \ge t_0$. 
\end{proof}

\black
\begin{example}\normalfont
In this example, {\it{(T1)}} is not satisfied and there is no attracting trapping region. 
The original basic two-dimensional Lotka-Volterra system has prey and predator species with population densities $x_1$ and $x_2$,
\begin{align}\label{std system}
\begin{split}
    \dt x_1 &= c_1x_1 + s_{12} x_1x_2, \\
    \dt x_2 &= c_2x_2 + s_{21} x_1x_2,
\end{split}
\end{align}
where $c_1>0>c_2$ and $s_{21}>0>s_{12}$. 
This is not a trophic system because $s_{11}=0$ so {\it{(T1)}} is not satisfied. There is no attracting trapping region because all solutions with $x_1,x_2>0$ are periodic except for the steady state $x_1=-\frac{c_2}{s_{21}}$, $x_2=-\frac{c_1}{s_{12}}$. Both numbers are positive.

The standard Lyapunov function for the system \eqref{std system} is 
\begin{align} 
    V(x_1,x_2):=s_{21}x_1+c_2\ln(x_1)- s_{12}x_2 - c_1 \ln(x_2),
\end{align}
where $x_1,x_2>0$. 
By direct calculation $\dt V(x_1,x_2) \equiv 0$. For each constant $v>0$, the set where $V(x_1,x_2)=v$ is a periodic orbit. Hence there are no globally attracting trapping regions. 

Since {\it{(T2)}} is satisfied, we conclude {\it{(T1)}} is essential for Theorem~\ref{thm:trap}.
\end{example}

\black
\section{Discussion}\label{sec discussion}

 Some of the readers of the preliminary drafts of this paper asked us which components are new, so here we summarize, including some points that were mentioned above. 
For Theorem~\ref {thm:trap}, the result is new. The technique of using a Lyapunov-like function for establishing a globally attracting trapping region – is not new, and we referred to Ed Lorenz’s use of such a function in his famous paper, \cite{lorenz1963deterministic}.

Theorem~\ref{thm:DieOut1} is also new. Most of the techniques are also new. Lyapunov functions using log terms are very common. We were motivated to use kernel vectors by papers by one of us, 
\cite{jahedi2022global,YorkeSaureJahedi}, papers that are not about differential equations. Die-out Lyapunov functions were motivated by our joint paper, \cite{akhavan2020population}
which featured a die-out Lyapunov function of the form $V=\frac{x_1}{x_2}$ where the space was 4-dimensional. 
Those papers are mentioned above. When dealing with the kernel of a matrix, our approach is to find all of the null vectors $\nu$ that have as few non-zero components as possible. Each yields a different $\Lambda_\nu$. All of these functions are needed to prove the theorem. A ``Team’’ of Lyapunov functions seems new. There are many papers on Lyapunov functions, but we have never seen anything like this. The idea of using each of those null vectors to create a die-out Lyapunov function seems to us unique.

\bigskip

{\bf When trajectories have no limit points in the open set \BF$\cP$.} 
{\black When a bounded trajectory $X(t)$ has some coordinates dying out as in Theorem~\ref{thm:DieOut1}, each of its limit points $X^*$ will have some coordinate(s) $x^*_i=0$. Hence, $X^*$ is in the closed set $\bcP$ but not in the open set $\cP$.}
As in Prop.~\ref{prop1}, assume $F$ is $C^1$ in the differential equation \eqref{DE}, where $F$ is defined on
an open set $U\subset \R^d$. Let  $X(t)=(x_1,\ldots,x_d)(t)\in U$ be a trajectory and let $V:U\to\R$ be differentiable.
Write $v(t)=V(X(t))$. Standard Lyapunov function theorems including Prop.~\ref{prop1} and the
Barbashin-Krasovskii-LaSalle Theorem (See \cite[p.~309]{alligood1996chaos}) conclude that $v(t)$ is monotonically decreasing. They 
make assumptions on $V$ that determine the behavior of $v(t)$. That is not the only approach. 

Taking the $\dt V$ definition Eq.~\eqref{dotV} a step further,
for a differential equation  \eqref{DE},
define the second Lyapunov derivative
\begin{align*}\label{ddtV}
 \ddt V(X) =\nabla \dt V(X) \cdot F(X) \text{ so that }\frac{d^2}{dt^2}V(X(t))=\ddt V(X(t)).
\end{align*}
Jacobi introduced this concept in 1840 for his ``stability criterion'' for the $N$-body problem (See \cite[p.118]{wilson1973lyapunov}).

Other higher Lyapunov derivatives are defined analogously. 
Some papers such as \cite{butz1969higher} and  \cite{ahmadi2011higher} consider the higher order derivatives $\dt V(X), \ddt{V}(X)$, and $\dddt{V}(X)$ to establish globally asymptotic stability.
Also, there are publications that discuss multiple Lyapunov functions but in a different context of the stability of a fixed point (see \cite{branicky1998multiple}, \cite{lakshmikantham2013vector}). 

The theorem in  
\cite{yorke1970theorem} uses the second derivative of $V(X)$, eliminating assumptions such as ``$V \ge 0$'' and ``$V$ is unbounded'' and ``$\dt V \le 0$'' while still obtaining the same conclusion about trajectories as in Prop.~\ref{prop1}.

The following striking result has conclusions modeled on Prop.~\ref{prop1}. We include it to demonstrate how even apparently weak conditions can result  in trajectories having no limit points in the interior of the domain of a differential equation.

\begin{prop}[{\black in the spirit of} \cite{yorke1970theorem}]\label{prop V double dot} 
Assume there exists a $C^1$ differential equation \eqref{DE} where $F$ is defined on
a simply connected open set $U\subset \R^d$ and $F$ is $C^1$. Assume \\
($A'_1$) there is a trajectory $X(t)=(x_1,\ldots,x_d)(t)\in U$ for all $t\ge 0$, and\\
($A'_2$) $V:U\to\R$ is $C^2$, and for each $X\in U$ either $\dt V(X) \ne 0$ or $\ddt V(X) \ne 0$.\\
Then\\
($B'_1$) $X(t)$ has no limit points in $U$ as $t\to\infty$.\\
($B'_2$) If furthermore $U=\cP$ and the trajectory is bounded, 
then $\displaystyle{\min_{1\le i\le d} x_i(t)}\to 0$ as $t\to\infty$.
\end{prop}
This result generalizes Prop.~\ref{prop1} since its requirement that $\dt V <0$ is a special case of condition $(A'_2)$ here, and in that we do not assume $U$ is invariant. 

This proposition seems to be virtually devoid of assumptions about $V$ that are useful in characterizing $v(t)=V(X(t))$ since $V(X(t))$ is allowed to have even an infinite number of local maxima and minima, as in the following toy example. 
There is even no assumption that $U$ is invariant.

For a trivial example consider the one-dimensional equation $x'=f(x)$. Assume the $C^1$ function $f:\R\to\R$ is never 0. Assume $X(t)$ is defined for all $t$. Then  for $V(x) := \sin(x)$, conditions ($A'_1$) and ($A'_2$) are satisfied since $\dt V(x) = \cos(x) f(x)$ and $\ddt V(x) = [-\sin(x) f(x) + \cos(x)f'(x)]f(x)$. When $\dt V(x) = 0$ we have $\ddt V(x) = -\sin(x)(f(x))^2 \neq 0$. Hence $(A'_2)$ is satisfied. 
By (iii), we conclude $X(t)$ has no limit points \ie, $X(t) \to \pm \infty$ as $t \to \infty$.

The proof is based on describing each connected component $\cQ$ of $\{X \in U:\dt V(X) =0\}$. It is shown that each such component $\cQ$ separates ${\cP}$ into two pieces, and each trajectory can pass through each $\cQ$ at most once.  Our conditions here are more general in some ways than those in \cite{yorke1970theorem},
but the proof there is easily adapted to our case. 
For example, here the domain  $U$ of the differential equation is only required to be simply connected; it is not assumed to be invariant.

{\bf Possible extensions}. 
The  die-out Lyapunov functions we use for 
Eq.~\eqref{equationX*-ndim} have the form
$\Lambda_{\nu} (X) := \sum_i \nu_i \ln x_i.$
It employs $\ln x_i$ for each $i$ because the left side of the equation involves $\frac{x_i'}{x_i}$ which is the derivative of $\ln x_i$.
So the reader may wish to extend the ideas and theorems here to the case where {\it some} or all of the  $\frac{x_i'}{x_i}$ are replaced by the simpler
$x_i'$. 
The corresponding terms in  die-out Lyapunov functions would be changed to 
$\nu_i x_i$ since the derivative of $x_i$ is $x_i'$. The domain of such a variable would be $\R$ instead of $(0,\infty)$. We leave it to the student or researcher to see what kinds of theorems can be created for such a hybrid system.

{\black Many systems of ordinary differential equations will have some coordinates dying out asymptotically. Most writers then write the equations for the remaining coord\-inates. Here, in some special cases, we have shown how to conclude that some coordinates die out. We hope that our effort here will encourage others to pursue systems where coordinates die out.}



\section{Acknowledgements}
The authors are grateful to  Thomas Breunung, Kathleen Hoffman, Bradford E. Peercy, and Yoshitaka Saiki whose careful reading and detailed comments identified significant shortcomings in early drafts and 
significantly  improved the clarity and organization of the manuscript. The authors also thank Timothy Sauer,  Sana Jahedi, Ioannis Kevrekidis for fruitful discussions. We appreciate that it can be quite difficult to find careful readers of a manuscript and that different readers with different perspectives will identify different problems.
NA was supported by NSF-NIH (DMS-NIGMS) Grant No. 1953423. 

\newpage
\begin{section}{Appendix}
\begin{table}[H]
\caption{List of variables in four-dimensional Lotka Volterra model, Fig.~\eqref{fig:3preds1prey-dynamics}, $i,j = 1, \ldots, 4$. 
These parameter choices are meant to provide a purely mathematical example to demonstrate the ideas of the paper. We make no claim that this represents an actual biological ecosystem.
} 
\label{table:Symbol2}
\centering 
\begin{tabular}{l c r} 
 Symbol & Variable Name& Initial Value
\\ [0.5ex] 
\hline
  $x_1$ &  Species (1) & 60   \\
  $x_2, x_3, x_4$ &  Species (2,3,4) & 40 \\
 \hline
 $x_i$ &  Species &  \\
  $c_j$ &  Net birth or death rate of species &    \\
  $s_{ij}$ &  Trophic coefficient &    \\
\end{tabular}  
\begin{tabular}{l c r} 
\hline
 Symbol & Parameter Name & Typical Value
\\ [0.5ex]  
\hline   
  $c_1$ & Birth rate of prey & $1.05$ \\
  $c_2$ & Death rate of the first predators & $0.29$  \\
  $c_3$ & Death rate of the first predators & $0.3$ \\
  $c_4$ & Death rate of the first predators & $0.31$\\
  $s_{11}$ & Rate of self-limiting of prey & $0.002$\\
  $s_{12}$ & First predator consumption & $0.008$ \\
  $s_{13}$ & Second predator consumption & $0.0075$\\
  $s_{14}$ & Third predator consumption & $0.006$ \\
  $s_{21}$ & Rate of change of first predator & $0.0023$\\
   & due to the presence of prey & \\
  $s_{31}$ & Rate of change of second predator & $0.0035$ \\
   & due to the presence of prey & \\
  $s_{41}$ & Rate of change of third predator & $0.003$\\
   & due to the presence of prey & \\
   \hline
   \end{tabular}  
\end{table}  
\end{section}
 \bibliographystyle{abbrvnat}
\bibliography{REF}

\end{document}